%% file: fisty131029ToArxiv.tex
\documentclass[12pt]{article}
\usepackage{amsmath}
\usepackage{amssymb}
\usepackage{mathrsfs} 
\usepackage{longtable}
\usepackage[dvipdfm]{graphicx}
\usepackage{amsthm}

\newtheorem{theorem}{Theorem}[section]
\newtheorem{lemma}[theorem]{Lemma}
\newtheorem{proposition}[theorem]{Proposition}
\newtheorem{conjecture}[theorem]{Conjecture}
\newtheorem{formula}[theorem]{Formula}
\newtheorem{corollary}[theorem]{Corollary}
\newtheorem{remark}[theorem]{Remark}

\newtheorem{example}[theorem]{Example}

\newcommand{\myvcenter}[1]{\ensuremath{\vcenter{\hbox{#1}}}}

\newcommand{\hako}[1]{\mbox{\boldmath$#1$}}

\title{A new multidimensional slow continued fraction algorithm
and stepped surface}

\date{\today}
\author{
{\sc Maki Furukado} \thanks{
College of Business Administration, 
Yokohama National University, 
79-4 Tokiwadai, Hodogaya-ku, Yokohama 
240-8501, Japan, 
E-mail: furukado@ynu.ac.jp
}
\and
{\sc Shunji Ito} \thanks{
Faculty of Science, 
Toho University, 
2-2-1 Miyama Funabashi-shi, Chiba 
274-8510, Japan, 
E-mail: shunjiito@gmail.com
}
\and
{\sc Asaki Saito} \thanks{
Faculty of Systems Information Science,
Future University Hakodate, 
116-2 Kameda Nakano-cho, Hakodate-shi, Hokkaido
041-8655, Japan, 
E-mail: saito@fun.ac.jp}
\and
{\sc Jun-ichi Tamura} \thanks{
Institute for Mathematics and Computer Science, Tsuda College, 
2-1-1 Tsudamachi, Kodaira-shi, Tokyo 
187-8577, Japan, 
E-mail: jtamura@tsuda.ac.jp}
\and
{\sc Shin-ichi Yasutomi} \thanks{
Faculty of Science, 
Toho University, 
2-2-1 Miyama Funabashi-shi, Chiba 
274-8510, Japan, 
E-mail: shinichi.yasutomi@sci.toho-u.ac.jp}
}
\begin{document}
\maketitle

\begin{abstract}

We give a new algorithm of slow continued fraction expansion related to any real cubic number field as a 2-dimensional version of the Farey map.
Using our algorithm, we can find the generators of dual substitutions (so-called tiling substitutions) for any stepped surface for any cubic direction.

\end{abstract}


%
\pagestyle{myheadings}
\thispagestyle{empty}
\baselineskip=15pt
\vskip 30pt

\section{Introduction}
The main topics of this paper are 

\noindent
(i)
to find a good algorithm of slow 
continued fraction expansion of dimension 
$2$ by which the expansion of $\bar{\alpha}=\left( \alpha_{1}, \alpha_{2} \right)$ is always expected to be periodic for any $\mathbb{Q}$-basis 
$\bar{\bar{\alpha}} = \left( 1, \alpha_{1}, \alpha_{2} \right)$ of 
an arbitrarily given
real cubic number field $K$  such that the
unimodular matrix 
$\mathcal{P}er$ coming from a period of the 
expansion has the minimal polynomial 
of a Pisot number
as its characteristic polynomial,

\noindent
and

\noindent
(ii)
to find a set of generators of dual substitutions for the 
stepped surface ${\mathscr S} \left( \bar{\bar \alpha} \right) $ of dimension $2$ for any ${\mathbb Q}$-basis 
$\bar{\bar \alpha}$ of $K$ and to give a finite description 
(or an effective construction) in terms of six dual
substitutions coming from the continued fraction expansion obtained
 by our algorithms, see Section \ref{sec:cont}, 
 \ref{sec:stepped} for notation.  

Notice that by the symmetry of the 
lattice $\mathbb{Z}^{s}$, we may assume that 
$\alpha_{0}$, $\alpha_{1}$, $\alpha_{2} >0$
with
$\alpha_{0} + \alpha_{1}+\alpha_{2}=1$.

The stepped surface $\mathscr S \left(
\bar{\bar \alpha} \right)$, 
$\left( 
\bar{\bar \alpha} = \left( \alpha_{0}, \alpha_{1}, \ldots,
\alpha_{s} \right)
\right)$ was introduced as a $s$-dimensional version of the
sturmian word, cf. \cite{IO, IO2, AI}.
The stepped surface of dimension $s=2$ is of special 
interest, since it is not only a geometrical object related to an aperiodic tiling, but also it has 
a connection with number theoretical problems related to simultaneous Diophantine approximations which are best possible up to constant, 
cf.\cite{IFHY}.  If $\bar{\bar \alpha} \in K^{s+1}$ is 
a $\mathbb{Q}$-basis of certain real algebraic number 
field $K$ of degree $s+1$, there is a 
beautiful connection between the stepped surface ${\mathscr S} \left( \bar{\bar \alpha} \right)$ and the continued fraction expansion of 
$\bar{\alpha} = \left( \alpha_{1}, \alpha_{2}, \ldots, \alpha_{s} \right) \in K^{s}$ provided that the continued fraction is periodic related to the Jacobi-Perron algorithm or the Brun algorithm (including the so-called modified Jacobi-Perron algorithm as its 
2-dimensional case). 
There appeared many papers concerning the construction 
of the stepped surface ${\mathscr S} \left( \bar{\bar \alpha} \right)$ for $\bar{\alpha}$ having a periodic continued fraction 
expansion, cf. \cite{IO, IO2, FIN, AI, ABI, F, BF}.

On the other hand, it has been a difficult problem to construct even a part of the stepped surface
${\mathscr S} \left( \bar{\bar \alpha} \right)$ for 
some
$\mathbb Q$-basis $\bar{\bar \alpha}$ of 
some real algebraic number field $K$, since there have not been a good deterministic algorithm to get a periodic 
continued fraction of $\bar{\alpha}$ for $s \geq 2$.  In a series of papers
(\cite{T, TY, TY4, TY2, TY3},\\
 \cite{TY5}) we made some good candidates of continued fraction algorithms for $1 \leq s \leq 4$, equipped with a value function $v$ (see 
Section \ref{sec:alg}) by which the algorithms become deterministic.  

The properties/conditions 
of the matrix 
$\mathcal{P}er$
mentioned in (i), i.e., 
the unimodularity, the Pisot property 
and the irreducibility are essential; in fact, under these conditions, the stepped surface becomes finitely descriptive, cf. Theorem 5 in \cite{F} due to T. Fernique.

Concerning our algorithm of continued fraction expansion of dimension $s=2$, which is expected to have the properties mentioned in (i) above, let us consider the case where $s=1$ for a while.  The algorithm
$\left(  \left[ 0, 1 \right], T, \varepsilon \right)$
 and its modified 
version are
considered by many authors (for example 
see \cite{I, IY}), where the transformation $T$ on the interval $\left[0,1 \right]$ is defined by 
\begin{align*}
&T(x):=\left\{
    \begin{array}{l}   
    \displaystyle \frac{x}{1-x} \quad 
\text{\ if \ }x \in I_0:= \left[ 0, \dfrac{1}{2} \right],\\
\\
\displaystyle\frac{2x-1}{x} \quad
\text{\ if \ }x\in I_1:= \left( \dfrac{1}{2}, 1 \right],\\
   \end{array}   
   \right.
   \end{align*}
$\varepsilon :[0,1]\to \{0,1\}$ with 
 $x\in I_{\varepsilon(x)}$.
 Let $x = \left[0; k_{1}, k_{2}, \ldots \right]$ be the simple continued fraction expansion, and let $F$ be the transformation on 
 $\left( 0, 1 \right]$ defined by $F\left(x \right) = \frac{1}{x} - \lfloor \frac{1}{x} \rfloor$.  
 Then, one can see that $T^{k_{1}+k_{2}}=F^{2} 
 \left( x \right)$ holds.  In this sense,
$\left(  \left[ 0, 1 \right], T, \varepsilon \right)$
 can be considered as a kind of slow continued fraction algorithm.  
$\left(  \left[ 0, 1 \right], T, \varepsilon \right)$
also has a connection with the Farey partition.  
The famous Lagrange's theorem says that,
if we consider the restriction $T_{K}$ of $T$
on 
$\left[ 0, 1 \right] \cap K$ for a real quadratic number field $K$, then every element $\alpha \in \left[ 0, 1 \right] \cap K$ becomes a periodic point of $T_{K}$, i.e., there exists $m \neq n \in 
{\mathbb Z}_{\geq 0}$ such that 
$T_{K}^{m} \left( \alpha \right) = T_{K}^{n} \left( \alpha \right)$.  
The expansion  obtained by the slow continued fraction algorithm 
$\left(  \left[ 0, 1 \right], T, \varepsilon \right)$ can be considered as an infinite word 
$\varepsilon \left( x \right)
\varepsilon \left( T \left( x \right) \right)
\varepsilon \left( T^{2} \left( x \right) \right) \cdots$ over an alphabet 
$\left\{ 0, 1 \right\}$.  Consequently, the generators of the dual substitutions on the stepped surface of dimension $1$ (the sturmian word of dimension $1$) consist of 
$\# \left\{ 0, 1 \right\}=2$ primitive substitutions.  We can 
extend the algorithm 
$\left(  \left[ 0, 1 \right], T, \varepsilon \right)$ to certain
multidimensional algorithms.  

In Section \ref{sec:algo}, we define a new deterministic algorithms of slow (additive) continued fraction of dimension $2$ equipped with a value function 
$v=v_{r} \left( \alpha, \beta, i, j \right)$.  The resulting expansion
by our algorithm can be considered as an infinite word 
over an alphabet $Ind$ given by 
\[
Ind:= \left\{
\left( i, j \right) \left |~ i, j \in \left\{ 0,1, 2 \right\}, \ i \neq j \right.
\right\},
\]
consequently the generators of the dual substitutions on
the stepped surface of dimension $2$ consist of $\# Ind=6$ 
primitive substitutions.

Theorem \ref{t-1} says the additivity of our algorithms. Theorems \ref{t1}, \ref{t2}, \ref{t3} give some admissibility conditions of the expansion obtained by our  algorithms. Proposition \ref{P1} says the convergence of the continued fractions.

Theorem \ref{P2} gives infinitely many examples of periodic expansions obtained by one of the algorithms.

In Section \ref{sec:cont}, we translate the expansion obtained by our algorithms into canonical representations of
continued fractions of dimension 2, cf. Theorem 
\ref{th:varepsilon}.
We also give reduction rules
by which we can make
acceleration of our continued fractions.

In Section \ref{sec:num}, 
we made the periodicity test for one of our algorithms
$\left( r=5/2 \right)$, cf. Table \ref{table:C}. 
We also gave an experiment, by using PC for pure cubic extensions 
$K={\Bbb Q} \left( \sqrt[3]{d} \right)$, $2\leq d \leq 10000$, $\sqrt[3]{d}\notin {\Bbb Q}$, which supports Conjecture \ref{con1}, 
that is a cubic version of Lagrange's theorem.
We also checked that 18797 continued fraction expansions
obtained by our algorithm with $r=5/2$
coming from the set $N_{15}$ are periodic, and the 
matrix $\mathcal{P}er$ always has the minimal polynomial of a Pisot number
as its characteristic polynomial.
Such a Pisot property was supported by another 
independent experiment for around 10000
continued fractions obtained by random generation of
totally real cubic number fields.

In Section \ref{sec:stepped}, we give some experiments which describe the generating process of the whole part of some stepped surfaces in terms of dual substitutions (or tiling substitutions) related to some real
cubic number field $K$ (including both totally real fields and fields having complex embeddings).

In Section \ref{sec:con}, we give two conjectures. Under these two conjectures together with Fernique's result
(Theorem 5 in \cite{F} mentioned above) 
we shall see that the stepped surface 
${\mathscr S}(\bar{\bar{\alpha}})$ becomes finitely descriptive by using only
six dual substitutions for any ${\Bbb Q}$-basis $\bar{\bar{\alpha}}$ of any given real cubic number field,
see Conclusion \ref{conclusion}.

\section{A new algorithm}
\label{sec:alg}
In what follows, $K$ denotes arbitrarily chosen fixed real cubic number field unless otherwise mentioned.
We put 
\[
\Delta_K:=
\left\{
\left( \alpha,\beta \right )\in K^2 ~\left|
\begin{array}{l}
1,\alpha,\beta \text{\ are linearly independent over ${\mathbb Q}$}, \\
 0<\alpha,\beta \ \text{and\ } \alpha+\beta<1 
 \end{array}
 \right.
 \right\}.
\]
We need a following lemma.
\label{sec:algo}
\begin{lemma} \label{aaa}
Let $p,q\in {\Bbb Z}^+$, $(p,q)=1$ 
and $p\not\equiv 0 \pmod{3}$.
Then,
\begin{align*}
\frac{|N(\alpha)|}{\alpha^{p/q}}=\frac{|N(\beta)|}{\beta^{p/q}}\ \text{implies\ }\alpha=\beta,
\end{align*}
for all $\alpha,\beta>0$ such that $\alpha,\beta\in K\backslash {\Bbb Q}$.
\end{lemma}

\begin{proof}
We suppose that 
$\displaystyle\frac{|N(\alpha)|}{\alpha^{p/q}}=\frac{|N(\beta)|}{\beta^{p/q}}$ holds
 for $\alpha,\beta\in K$ with $\alpha,\beta\notin {\Bbb Q}$ and $\alpha,\beta>0$.
Let $\zeta=\alpha/\beta$. Then, we have $\zeta^{p}=|N(\zeta)|^q$.
Therefore, we see that  $N(\zeta)^{p}=|N(\zeta)|^{3q}$.
Since $N(\zeta)$ is a rational number and $p$ is not divisible by $3$,
we see $|N(\zeta)|=1$. Thus, we have $\zeta^{p}=1$, which implies $\zeta=\pm 1$.
Since $\alpha,\beta>0$, $\zeta=1$.
\end{proof}

We put 
\begin{align*}
Ind:=\{(i,j)|\ i,j\in \{0,1,2\},i\ne j\}.
\end{align*}

We denote domains $\Delta$ and $\triangle(i,j)$ for $(i,j)\in Ind$ 
by 
\begin{align*}
&\Delta:=\{(x,y)\in {\Bbb R}^2~|~x,y\geq 0, x+y\leq 1\},\\
&\triangle(1,2):=\{(x,y)\in \Delta~|~x\geq y\},\\
&\triangle(2,1):=\{(x,y)\in \Delta~|~x\leq y\},\\
&\triangle(0,1):=\{(x,y)\in \Delta~|~2x+y-1\leq 0\},\\
&\triangle(1,0):=\{(x,y)\in \Delta~|~2x+y-1\geq 0\},\\
&\triangle(0,2):=\{(x,y)\in \Delta~|~x+2y-1\leq 0\},\\
&\triangle(2,0):=\{(x,y)\in \Delta~|~x+2y-1\geq 0\}
\end{align*}
(see Figure \ref{fig:DomainDeltas}).

\begin{figure*}[hbpt]
\begin{center}
\includegraphics[width=13cm]{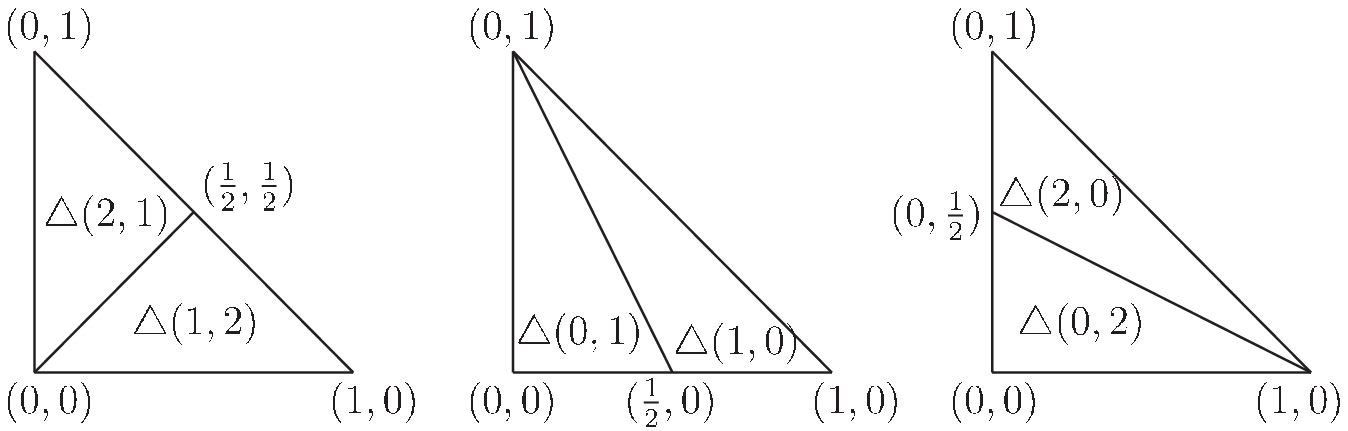}
\end{center}
\caption{The domains $\triangle 
\left( i, j \right)$, 
$\left( i, j \right) \in Ind.$}
\label{fig:DomainDeltas}
\end{figure*}

For each $\left(i, j \right) \in Ind$, let us introduce the  maps $T_{\left(i, j \right)}: \triangle \left( i, j \right) \rightarrow \Delta$ as follows:
\begin{align*}
&T_{\left(1,2 \right)}(x,y):=\displaystyle\left(\frac{x-y}{1-y},\frac{y}{1-y}\right),\\
&T_{\left( 2,1 \right)}(x,y):=\displaystyle\left(\frac{x}{1-x},\frac{y-x}{1-x}\right),\\ 
&T_{\left( 0, 1 \right)}(x,y):=\displaystyle\left(\frac{x}{1-x},\frac{y}{1-x}\right),\\
&T_{\left( 1,0 \right)}(x,y):=\displaystyle\left(\frac{2x+y-1}{x+y},\frac{y}{x+y}\right),\\
&T_{\left( 0, 2 \right)}(x,y):=\displaystyle\left(\frac{x}{1-y},\frac{y}{1-y}\right),\\
&T_{\left( 2, 0 \right)}(x,y):=\displaystyle\left(\frac{x}{x+y},\frac{x+2y-1}{x+y}\right).
\end{align*}

We define the value $v_r(\alpha,\beta,i,j)$ for $r\in {\Bbb R}^+$, $(\alpha,\beta)\in \Delta_K$ and $i,j\in \{0,1,2\}$
with $i\ne j$ as follows:
\[
v_r(\alpha,\beta,i,j):=
    \begin{cases}   
    \dfrac{|\alpha^{r}\beta^{r}|}{|N(\alpha)N(\beta)|}
&\text{ if  $\{i,j\}=\{1,2\}$},\\[10pt]
    \dfrac{|\alpha^{r}(1-\alpha-\beta)^{r}|}{|N(\alpha)N(1-\alpha-\beta)|}
&\text{ if  $\{i,j\}=\{0,1\}$},\\[10pt]
   \dfrac{|\beta^{r}(1-\alpha-\beta)^{r}|}{|N(\beta)N(1-\alpha-\beta)|},\ 
&\text{ if  $\{i,j\}=\{0,2\}$}.
   \end{cases}   
\]
In what follows,
we suppose that $r=p/q$ with
$p,q\in {\Bbb Z}^+$, $(p,q)=1$ and $p\not\equiv 0$ mod $3$ in this paper.

It follows from Lemma \ref{aaa} that the element $(i_0,j_0)\in Ind$
is uniquely determined by $v_r(\alpha,\beta,i_0,j_0)=\max \{ v_r(\alpha,\beta,i,j)\}$.
We define $\varepsilon(\alpha,\beta)=\varepsilon_K(\alpha,\beta)$ for $(\alpha,\beta)\in \Delta_K$ by
\[
\varepsilon(\alpha,\beta):=
    \begin{cases}
(1,2) &\text{if $\{i_0,j_0\}=\{1,2\}$ and $\left( \alpha,\beta \right)\in \triangle(1,2)$},\\
(2,1) &\text{if $\{i_0,j_0\}=\{1,2\}$ and  $(\alpha,\beta)\in \triangle(2,1)$},\\
(0,1) &\text{if  $\{i_0,j_0\}=\{0,1\}$ and $(\alpha,\beta)\in \triangle(0,1)$},\\
(1,0) &\text{if  $\{i_0,j_0\}=\{0,1\}$ and $(\alpha,\beta)\in \triangle(1,0)$},\\
(0,2) &\text{if  $\{i_0,j_0\}=\{0,2\}$ and $(\alpha,\beta)\in \triangle(0,2)$},\\
(2,0) &\text{if $\{i_0,j_0\}=\{0,2\}$ and $(\alpha,\beta)\in \triangle(2,0)$}.
    \end{cases}   
\]
Notice that $\varepsilon(\alpha,\beta)$
is well-defined since $1,\alpha,\beta$ is linearly independent over ${\Bbb Q}$.
We define the transformation $T=T_K=T_{K,r}$ on $\Delta_K$ by
\begin{align*}
\ T(\alpha,\beta):=T_{(i_0,j_0)}(\alpha,\beta) \ \text{if  $\varepsilon(\alpha,\beta)=(i_0,j_0)$}.
\end{align*}

Thus, we have seen that an algorithm $(\Delta_K,T,\varepsilon)$ can be defined.
We put
\begin{align}\label{d2}
&A_{(1,2)}:=
\begin{pmatrix}
1& 0& 1\\
0& 1& 1\\
0& 0& 1
\end{pmatrix}
,
\quad
A_{(2,1)}:=
\begin{pmatrix}
1& 1& 0\\
0& 1& 0\\
0& 1& 1
\end{pmatrix}
,\\
&A_{(0,1)}:=
\begin{pmatrix}
1& 1& 0\\
0& 1& 0\\
0& 0& 1
\end{pmatrix}
,
\quad
A_{(1,0)}:=
\begin{pmatrix}
2& -1& -1\\
1& 0& -1\\
0& 0& 1
\end{pmatrix}
,\\
&
A_{(0,2)}:=
\begin{pmatrix}
1& 0& 1\\
0& 1& 0\\
0& 0& 1
\end{pmatrix}
,
\quad
A_{(2,0)}:=
\begin{pmatrix}
2& -1& -1\\
0& 1&  0\\
1& -1& 0
\end{pmatrix}
.
\end{align}
For $n\in {\Bbb Z}_{\geq 0}$ we define $(\alpha_n,\beta_n)=T^n(\alpha,\beta)$. 
For $n\in {\Bbb Z}_{>0}$ we define
\begin{align}\label{eqqm}
M_n(\alpha,\beta)&=
\begin{pmatrix}
p''_n(\alpha,\beta)& p'_n(\alpha,\beta)& p_n(\alpha,\beta)\\
q''_n(\alpha,\beta)& q'_n(\alpha,\beta)& q_n(\alpha,\beta)\\
r''_n(\alpha,\beta)& r'_n(\alpha,\beta)& r_n(\alpha,\beta)\\
\end{pmatrix}
:=A_{\varepsilon(\alpha_0,\beta_0)}\cdots A_{\varepsilon(\alpha_{n-1},\beta_{n-1})}S,	
\end{align}
where 
\begin{align*}
S:=
\begin{pmatrix}
1& 1& 1\\
0& 1& 0\\
0& 0& 1
\end{pmatrix}.
\end{align*}
Then, we have following Theorem  \ref{t-1}.
\begin{theorem} \label{t-1}
For $(\alpha,\beta)\in \Delta_K$ and $n\in {\Bbb Z}_{\geq 0}$
\begin{align}
\text{$\varepsilon(\alpha_n,\beta_n)=(1,2)$}
\Rightarrow
\begin{cases}
(p_{n+1},q_{n+1},r_{n+1})=
(p_{n},q_{n},r_{n})+(p'_{n},q'_{n},r'_{n}),\label{bbb}\\
(p'_{n+1},q'_{n+1},r'_{n+1})=(p'_{n},q'_{n},r'_{n}),\\
(p''_{n+1},q''_{n+1},r''_{n+1})=(p''_{n},q''_{n},r''_{n}),\\
\end{cases}
\\
\text{$\varepsilon(\alpha_n,\beta_n)=(2,1)$}
\Rightarrow
\begin{cases}
(p_{n+1},q_{n+1},r_{n+1})=
(p_{n},q_{n},r_{n}),\label{bb2}\\
(p'_{n+1},q'_{n+1},r'_{n+1})=(p'_{n},q'_{n},r'_{n})+(p_{n},q_{n},r_{n}),\\
(p''_{n+1},q''_{n+1},r''_{n+1})=(p''_{n},q''_{n},r''_{n}),\\
\end{cases}
\\
\text{$\varepsilon(\alpha_n,\beta_n)=(0,1)$}
\Rightarrow
\begin{cases}
(p_{n+1},q_{n+1},r_{n+1})=
(p_{n},q_{n},r_{n}),\\
(p'_{n+1},q'_{n+1},r'_{n+1})=(p'_{n},q'_{n},r'_{n})+(p''_{n},q''_{n},r''_{n}),\\
(p''_{n+1},q''_{n+1},r''_{n+1})=(p''_{n},q''_{n},r''_{n}),\\
\end{cases}
\\
\text{$\varepsilon(\alpha_n,\beta_n)=(1,0)$}
\Rightarrow
\begin{cases}
(p_{n+1},q_{n+1},r_{n+1})=
(p_{n},q_{n},r_{n}),\\
(p'_{n+1},q'_{n+1},r'_{n+1})=(p'_{n},q'_{n},r'_{n}),\\
(p''_{n+1},q''_{n+1},r''_{n+1})=(p''_{n},q''_{n},r''_{n})+(p'_{n},q'_{n},r'_{n}),\\
\end{cases}
\\
\text{$\varepsilon(\alpha_n,\beta_n)=(0,2)$}
\Rightarrow
\begin{cases}
(p_{n+1},q_{n+1},r_{n+1})=
(p_{n},q_{n},r_{n})+(p''_{n},q''_{n},r''_{n}),\label{bb3}\\
(p'_{n+1},q'_{n+1},r'_{n+1})=(p'_{n},q'_{n},r'_{n}),\\
(p''_{n+1},q''_{n+1},r''_{n+1})=(p''_{n},q''_{n},r''_{n}),
\end{cases}
\\
\text{$\varepsilon(\alpha_n,\beta_n)=(2,0)$}
\Rightarrow
\begin{cases}
(p_{n+1},q_{n+1},r_{n+1})=
(p_{n},q_{n},r_{n}),\\
(p'_{n+1},q'_{n+1},r'_{n+1})=(p'_{n},q'_{n},r'_{n}),\\
(p''_{n+1},q''_{n+1},r''_{n+1})=(p''_{n},q''_{n},r''_{n})+(p_{n},q_{n},r_{n}).
\end{cases}
\end{align}
\end{theorem}

\begin{proof}
Let $\varepsilon(\alpha_n,\beta_n)=(1,2)$.
Then, we get 
\begin{eqnarray*}
\lefteqn{
\begin{pmatrix}
p''_{n+1}(\alpha,\beta)& p'_{n+1}(\alpha,\beta)& p_{n+1}(\alpha,\beta)\\
q''_{n+1}(\alpha,\beta)& q'_{n+1}(\alpha,\beta)& q_{n+1} (\alpha,\beta)\\
r''_{n+1}(\alpha,\beta)& r'_{n+1}(\alpha,\beta)& r_{n+1}(\alpha,\beta)
\end{pmatrix}
}\\
&=&
A_{\varepsilon(\alpha_{0},\beta_{0})}\ldots	
A_{\varepsilon(\alpha_{n},\beta_{n})}S	
=
A_{\varepsilon(\alpha_{0},\beta_{0})}\ldots	
A_{\varepsilon(\alpha_{n-1},\beta_{n-1})}
\begin{pmatrix}
1&1&2\\
0&1&1\\
0&0&1
\end{pmatrix}
\\
&=&
\begin{pmatrix}
p''_{n}(\alpha,\beta)& p'_{n}(\alpha,\beta)& p_{n}(\alpha,\beta)+p'_{n}(\alpha,\beta)\\
q''_{n}(\alpha,\beta)& q'_{n}(\alpha,\beta)& q_{n}(\alpha,\beta)+q'_{n}(\alpha,\beta)\\
r''_{n}(\alpha,\beta)& r'_{n}(\alpha,\beta)& r_{n}(\alpha,\beta)+r'_{n}(\alpha,\beta)\\
\end{pmatrix}
.
\end{eqnarray*}
Thus, we have (\ref{bbb}).
We have (\ref{bb2})-(\ref{bb3}) in the similar manner.
\end{proof}

We shall give some theorems concerning matrices $A_{(i,j)}$.
We need some definitions.  Let
$\mathbb{P}_{2} \left( \mathbb{R} \right)$ 
be the projective space of dimension 2 over $\mathbb{R}$, i.e., \[
\mathbb{P}_{2} \left( \mathbb{R} \right) =
\left( \mathbb{R}^{3} \setminus \left\{ \bar{\bar{0}} \right\} \right) / \sim,
\]
where $\sim$ is an equivalence  relation defined by 
\[
\bar{\bar{x}} \sim \bar{\bar{y}} 
\quad
\left( \bar{\bar{x}}, \bar{\bar{y}} \in \mathbb{R}^{3} \setminus
\left\{ \bar{\bar{0}} \right\} \right)
\Leftrightarrow 0 \neq \exists c \in \mathbb{R}
\mbox{ such that } \bar{\bar{x}} = c \bar{\bar{y}}.
\]
We mean by $f: X -\rightarrow Y$ a "map" from a set $X$ to a set $Y$ with some exceptional elements $x \in X$ 
for which the value $f (x)$ is not defined.
For a matrix $A \in M_{3} \left( \mathbb{R} \right)$ 
(which denotes the set of $3 \times 3$ matrices
of real components), 
a map 
\[
A^{\mbox{\it proj}}: \mathbb{P}_{2} \left( \mathbb{R} \right)
- \rightarrow \mathbb{P}_{2} \left( \mathbb{R} \right)
\]
can be defined by 
\[
A^{\mbox{\it proj}} \left( \kappa \left( \bar{\bar{x}} \right) \right):= \kappa A \left( \bar{\bar{x}} \right),
\]
where 
$\kappa \left( \bar{\bar{x}} \right):=
\left\{
c \bar{\bar{x}} ~\left|~ 0 \neq c \in \mathbb{R}
\right.
\right\} \in \mathbb{P}_{2} \left( \mathbb{R} \right)$.
Notice that the map $A^{\mbox{\it proj}}$ is well-defined and 
\begin{equation}
\label{eq:proj1}
\left( A B \right)^{\mbox{\it proj}} = A^{\mbox{\it proj}} B^{\mbox{\it proj}}
\end{equation}
holds for $A, B \in M_{3} \left( \mathbb{R} \right)$.
We define two maps
$\pi: \mathbb{P}_{2} \left( \mathbb{R} \right) - \rightarrow 
\mathbb{R}^{2}$
and
$\iota: \mathbb{R}^{2}  \rightarrow 
\mathbb{R}^{3}$ by 
\begin{eqnarray*}
\pi \left( \kappa \left( \bar{\bar{x}} \right) \right) & := & 
\frac{1}{x^{(0)}} { x^{(1)} \choose x^{(2)}}
\quad \mbox{ for } 
\bar{\bar{x}} = ~^{t} \left( x^{(0)}, x^{(1)}, x^{(2)} \right) 
\in \mathbb{R}^{3},\\
\iota \left( \bar{x}  \right) & := & ~^{t} (1, x^{(1)}, x^{(2)}) 
\quad \mbox{ for } 
\bar{x}=(x^{(1)}, x^{(2)} )  
\in \mathbb{R}^{2}.
\end{eqnarray*}
Then, the linear fractional map $A^{\mbox{\it frac}}:
\mathbb{R}^{2} - \rightarrow \mathbb{R}^{2}$ 
$\left( A \in M_{3} \left( \mathbb{R} \right) \right)$ can be defined by 
\[
A^{\mbox{\it frac}} \left( \bar{x} \right) := 
\pi A^{\mbox{\it proj}} \kappa \iota \left( \bar{x} \right).
\]
Notice that this map is also well-defined
and the diagram (see Figure \ref{fig:diagram}) commutes. 
\begin{figure*}[hbtp]
\begin{center}
\includegraphics[width=5cm]{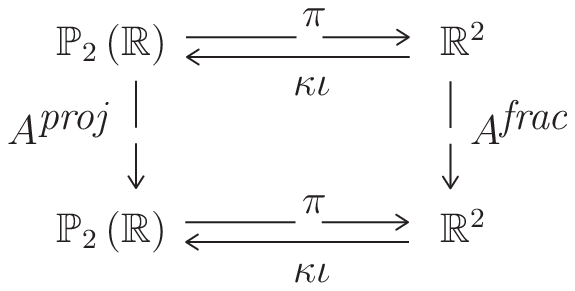}
\end{center}
\caption{
The commutative diagram with respect to  $A^{proj}$ and $A^{frac}$. 
}
\label{fig:diagram}
\end{figure*}

Hence, in view of (\ref{eq:proj1}), we get the following.

\begin{lemma}

\label{lem:AB}
\[
\left( AB \right)^{\mbox{\it frac}} = 
A^{\mbox{\it frac}} B^{\mbox{\it frac}}
\quad
\left( A, B \in M_{3} \left( \mathbb{R} \right) \right).
\]
\end{lemma}

We easily see following Theorem \ref{t0}. 

\begin{theorem} \label{t0}
For each $(i,j)\in Ind$, 
$A_{(i,j)}^{\mbox{\it frac}}\circ T_{
\left( i, j \right)}$
(resp. $T_{\left( i, j \right)}
\circ A_{(i,j)}^{\mbox{\it frac}}$)
is an identity map on $\triangle(i,j)$ (resp., $\Delta)$.
\end{theorem}

For each $(\alpha,\beta)\in \Delta_K$ and $n\in {\Bbb Z}_{\geq 0}$
we define $\delta_n(\alpha,\beta)$ by
the set of all inner points in a triangle with the edge points  
$(q_n/p_n,r_n/p_n)$, $(q'_n/p'_n,r'_n/p'_n)$
and $(q''_n/p''_n,r''_n/p''_n)$.
From Lemma \ref{lem:AB} and Theorem \ref{t0} we have following:

\begin{theorem} \label{t1}
Let $(\alpha,\beta)\in \Delta_K$. For each $n\in {\Bbb Z}_{\geq 0}$, 
$(\alpha,\beta)\in \delta_n(\alpha,\beta)$ holds.
\end{theorem}

The following theorem describes an admissibility of the sequence
 $\{\varepsilon(\alpha_n,\beta_n)\}_{n=0,1,\ldots}$
obtained by the algorithm 
$(\Delta_K,T,\varepsilon)$.

\begin{theorem} \label{t2}

Let $(\alpha,\beta)=(\alpha_{0},\beta_{0})$ and $n\in {\Bbb Z}_{\geq 0}$.
Then, both
\begin{align*}
&\varepsilon(\alpha_{n+1},\beta_{n+1})\ne (i',\theta(\varepsilon(\alpha_n,\beta_n)))\\
&\text{and}\\
&\varepsilon(\alpha_{n+1},\beta_{n+1})\ne (\theta(\varepsilon(\alpha_n,\beta_n)),i')
\end{align*}
hold, where $\varepsilon(\alpha_{n},\beta_{n})=(i',j')$ and
$\theta(i,j):=k\in \{0,1,2\}$ 
with $k\ne i$ and $k\ne j$, 
so that there are 12 forbidden words
$\varepsilon(\alpha_{n},\beta_{n})\varepsilon(\alpha_{n+1},\beta_{n+1})$ (see Table \ref{table:forbidden}).
\end{theorem}

\begin{table}[hbtp]
\caption{
The forbidden words of length $2$.} 
\label{table:forbidden}
\begin{center}
\input{table1.tex}
\end{center}
\end{table}

\begin{proof}
First, we suppose that $\varepsilon(\alpha_n,\beta_n)=(2,1)$.
Let
\(
\gamma_{n}= 1 -\alpha_{n}-\beta_{n}.
\)
From the definition of the {\it value} function,
we have
\begin{equation}
 \label{eq:22}
\frac{ \left| \alpha_{n}^{r} \beta_{n}^{r}\right|}{
 \left| N \left( \alpha_{n}  \right) N \left(  \beta_{n} \right) \right|}
>
\frac{ \left|  
\alpha_{n}^{r} \gamma_{n}^{r} \right|}{
 \left| 
 N \left( \alpha_{n} \right) N \left( \gamma_{n} \right) 
 \right|}, 
 \quad
\frac{ \left| \alpha_{n}^{r}  \beta_{n}^{r} \right|}{
 | N \left( \alpha_{n}  
  \right) N \left(  \beta_{n} \right) |}
>
\frac{ \left| \beta_{n}^{r} \gamma_{n}^{r}  \right|}{
 | N \left( \beta_{n}  \right) N \left(  \gamma_{n}  \right) |}.
\end{equation}
(\ref{eq:22}) is equivalent to
\begin{equation}
\frac{ \left| \beta_{n} \right|^{r}}{
 \left| N \left( \beta_{n} \right) \right|}
>
\frac{ \left| \gamma_{n}  \right|^{r}}{
 \left| N \left( \gamma_{n}  \right) \right|}, 
 \quad
\frac{ \left| \alpha_{n}  \right|^{r}}{
 \left| N \left( \alpha_{n}  \right) \right|}
>
\frac{ \left| \gamma_{n}  \right|^{r}}{
 \left| N \left( \gamma_{n}  \right) \right|}.
 \label{eq:2}
\end{equation}
Moreover, let $\left( \alpha_{n+1}, \beta_{n+1} \right)$ and
$\gamma_{n+1}$ be
\begin{align*}
\left( \alpha_{n+1}, \beta_{n+1} \right)
&
:= T \left( \alpha_{n}, \beta_{n} \right)
=
\left(
\frac{\alpha_{n}}{1-\alpha_{n}}, 
\frac{\beta_{n}-\alpha_{n}}{1-\alpha_{n}}
\right), \\
\gamma_{n+1} 
&
:=1-\alpha_{n+1}-\beta_{n+1} = 
\frac{{\gamma}_{n}}{1-\alpha_{n}}.
\end{align*}
We suppose that
$\varepsilon(\alpha_{n+1},\beta_{n+1})=(2,\theta(\varepsilon(\alpha_n,\beta_n)))$
or $\varepsilon(\alpha_{n+1},\beta_{n+1})=$\\
$(\theta(\varepsilon(\alpha_n,\beta_n)),2)$, i.e., 
$\varepsilon(\alpha_{n+1},\beta_{n+1})\in \{(2,0),(0,2)\}$.
Then, from the analogous above discussion, we get
\begin{align*}
\frac{ \left| \beta_{n+1}^{r} 
\gamma_{n+1}^{r} \right|}{
 \left| 
N \left( \beta_{n+1} \right) N \left( \gamma_{n+1} \right) \right|}
&>
\frac{ \left| \alpha_{n+1}^{r} \beta_{n+1}^{r}\right|}{
 \left| N \left( \alpha_{n+1} \right) 
 N \left( \beta_{n+1} \right) \right|}, \\
\frac{ \left| 
\beta_{n+1}^{r} \gamma_{n+1}^{r} \right|}{
 \left| 
 N \left( \beta_{n+1} \right) N \left( \gamma_{n+1} \right) \right|}
&>
\frac{ \left| 
\alpha_{n+1}^{r} \gamma_{n+1}^{r}  \right|}{
 \left| 
 N \left( \alpha_{n+1} \right) N \left( \gamma_{n+1} \right) \right|},
\end{align*}
i.e., 
\begin{equation}
\frac{ \left| \gamma_{n+1}^{r} \right|}{
 \left| N \left( 
 \gamma_{n+1} \right) \right|}
>
\frac{ \left|\alpha_{n+1}^{r}  \right|}{
 \left| N \left( \alpha_{ n+1}  \right) \right|}, 
 \quad
\frac{ \left| \beta_{n+1}^{r}  \right|}{
 \left| N \left( \beta_{ n+1}  \right) \right|}
>
\frac{ \left| \alpha_{n+1}^{r}  \right|}{
 \left| N \left( \alpha_{n+1}  \right) \right|}.
 \label{eq:3}
\end{equation}
From $N \left( \alpha \beta \right)=N \left( \alpha \right)
N \left( \beta \right)$, (\ref{eq:3}) is written by
\[
\frac{ \left| 
\dfrac{\gamma_{n}}{1-\alpha_{n}} \right|^{
r}}{
 \left| N \left( 
 \dfrac{\gamma_{n}}{1-\alpha_{n}} \right) \right|}
>
\frac{ \left| \dfrac{\alpha_{n}}{1-\alpha_{n}} \right|^{r}}{
 \left| N \left( \dfrac{\alpha_{ n}}{1-\alpha_{ n}} \right) \right|},
 \quad
\frac{ \left| 
\dfrac{\beta_{n}-\alpha_{n}}{1-\alpha_{n}} \right|^{r}}{
 \left| N \left( 
 \dfrac{
 \beta_{n}
 -\alpha_{n}}{1-\alpha_{n}} 
 \right) \right|}
>
\frac{ \left| 
                  \dfrac{
                             \alpha_{n}}{
                             1-\alpha_{n}} 
\right|^{r}}{
\left| N \left( 
\dfrac{\alpha_{n}}{1-\alpha_{n}} \right) \right|},
 \]
i.e., 
\begin{equation}
\frac{ \left| 
\gamma_{n} \right|^{r}}{
 \left| N \left( 
 \gamma_{n} \right) \right|}
>
\frac{ \left| \alpha_{n}  \right|^{r}}{
 \left| N \left( \alpha_{n}  \right) \right|}, 
 \quad
\frac{ \left| 
\beta_{n} - \alpha_{n}  \right|^{r}}{
 \left| N \left( 
 \beta_{n} - \alpha_{n}  \right) \right|}
>
\frac{ \left| \alpha_{n}  \right|^{r}}{
 \left| N \left( \alpha_{n}  \right) \right|}.
 \label{eq:4}
\end{equation}
But, (\ref{eq:4}) contradicts  (\ref{eq:2}).
Therefore,  
\begin{align*}
\varepsilon(\alpha_{n+1},\beta_{n+1})\ne(2,\theta(\varepsilon(\alpha_n,\beta_n))),
(\theta(\varepsilon(\alpha_n,\beta_n)),2)
\end{align*}
holds.
The other cases can be proved analogously.
\end{proof}

Theorem \ref{t2} says that there are some forbidden words $\left( i, j \right)$ in the sequences $\{\varepsilon(\alpha_n,\beta_n)\}_{n=0}^{\infty}$ obtained by the algorithm $(\Delta_K,T,\varepsilon)$. On the other hand, there exists $(\alpha,\beta)\in \Delta_K$
such that
the other words of length $2$ except for the forbidden words eventually appear 
in the 
sequences $\{\varepsilon(\alpha_n,\beta_n)\}_{n=0}^{\infty} $ for any real cubic field $K$: 

\begin{theorem} \label{t3}
For  $(i,j),(k,l)\in Ind,$ if $(k,l)\ne (i,\theta(i,j))$ and
$(k,l)\ne (\theta(i,j),i)$, then
there exists $(\alpha,\beta)\in \Delta_K$ such that
$\varepsilon_K(\alpha,\beta)=(i,j)$ and $\varepsilon_K\left(T_K(\alpha,\beta)\right)=(k,l)$.
\end{theorem}

\begin{proof}
Let $\lambda$ be the root of $x^3-5x+1$ with $\lambda>1$.
Let $K_1={\Bbb Q}(\lambda)$.
We note that $K_1$ is a totally real cubic field.
By the direct calculation one can check  Table 
\ref{table:A} given below:

\begin{table}[hbtp]
\caption{(totally real case) The words of length 
$2$ which are 
not forbidden by Table 
\ref{table:forbidden}
eventually occur.}
\label{table:A}
\begin{center}
\input{table2.tex}
%
%
\end{center}
\end{table}

We suppose that $K$ is a totally real cubic field.
Let $\,_1\rho_K,\,_2\rho_K$ be  distinct    embeddings from $K$ to ${\Bbb R}$
over  ${\Bbb Q}$ different from the non trivial embedding. 
We define a mapping $\rho_K$ from $\Delta_K$ to $\Delta\times {\Bbb R}^4$ as follows. For $(\alpha,\beta)\in \Delta_K$,    
\begin{align*}
\rho_K(\alpha,\beta):=(\alpha,\beta,\,_1\rho_K(\alpha),\,_1\rho_K(\beta),\,_2\rho_K(\alpha),\,_2\rho_K(\beta)).
\end{align*}
Then, it is not difficult to see that
$\rho_K(\Delta_K)$ is dense in $\Delta\times {\Bbb R}^4$ by virtue of algebraic number theory (for example see  Chapter 1 in \cite{N}).
We put $(\alpha',\beta')=(5/39+7\lambda/39-2\lambda^2/39,2/39-5\lambda/39+7\lambda^2/39)$
which is  in Table \ref{table:A}. 
Then, by the analogous  discussion in Theorem \ref{t2}, one  can show  
\begin{equation}
\frac{ \left| \beta' \right|^{r}}{
 \left| N \left( \beta' \right) \right|}
>
\frac{ \left| \gamma'  \right|^{r}}{
 \left| N \left( \gamma'  \right) \right|}, 
 \quad
\frac{ \left| \alpha' \right|^{r}}{
 \left| N \left( \alpha'  \right) \right|}
>
\frac{ \left| \gamma'  \right|^{r}}{
 \left| N \left( \gamma'  \right) \right|}
 \label{eq:x1}
\end{equation}
and
\begin{equation}
\beta'>2\alpha', 
\quad
\frac{ \left| 
\beta' - \alpha'  \right|^{r}}{
 \left| N \left( 
 \beta' - \alpha'  \right) \right|}
>
\frac{ \left| \gamma'  \right|^{r}}{
 \left| N \left( \gamma'  \right) \right|},
 \label{eq:x2}
\end{equation}
where $\gamma'=1-\alpha'-\beta'$.
We see that there exists $\delta>0$ such that
 $|x-\alpha'|<\delta$, $|y-\beta'|<\delta$,
 $|x'-\,_1\rho_{K_1}(\alpha')|<\delta$,
$|y'-\,_1\rho_{K_1}(\beta')|<\delta$,\\
 $|x''-\,_2\rho_{K_1}(\alpha')|<\delta$ and
$|y''-\,_2\rho_{K_1}(\beta')|<\delta$ implies 
\begin{equation}
\frac{ \left| y \right|^{r}}{
 \left| yy'y'' \right|}
>
\frac{ \left| z  \right|^{r}}{
 \left| zz'z'' \right|}, 
 \quad
\frac{ \left| x \right|^{r}}{
 \left| xx'x'' \right|}
>
\frac{ \left| z  \right|^{r}}{
 \left| zz'z'' \right|}
 \label{eq:x3}
\end{equation}
and
\begin{equation}
y>2x,
\quad
\frac{ \left| 
y - x  \right|^{r}}{
 \left| (y-x)(y'-x')(y''-x'') \right|}
>
\frac{ \left| z  \right|^{r}}{
 \left| zz'z'' \right|},
 \label{eq:x4}
\end{equation}
where $z=1-x-y$, $z'=1-x'-y'$ and $z''=1-x''-y''$
for every $(x,y,x',y',x'',y'')\in \Delta\times {\Bbb R}^4$.
Since $\rho_K(\Delta_K)$ is dense in $\Delta\times {\Bbb R}^4$, 
there exists an element $(\alpha,\beta)\in \Delta_K$
such that
\begin{equation}
\frac{ \left| \beta \right|^{r}}{
 \left| N \left( \beta \right) \right|}
>
\frac{ \left| \gamma  \right|^{r}}{
 \left| N \left( \gamma  \right) \right|}, 
 \quad
\frac{ \left| \alpha \right|^{r}}{
 \left| N \left( \alpha  \right) \right|}
>
\frac{ \left| \gamma  \right|^{r}}{
 \left| N \left( \gamma  \right) \right|}
 \label{eq:x5}
\end{equation}
and
\begin{equation}
\beta>2\alpha,
\quad
\frac{ \left| 
\beta - \alpha  \right|^{r}}{
 \left| N \left( 
 \beta - \alpha  \right) \right|}
>
\frac{ \left| \gamma  \right|^{r}}{
 \left| N \left( \gamma  \right) \right|},
 \label{eq:x6}
\end{equation}
where $\gamma=1-\alpha-\beta$.
From (\ref{eq:x5}) and (\ref{eq:x6}), it follows  
$\varepsilon_{K}(\alpha,\beta)=(2,1)$ and
$\varepsilon_{K}(T_{K}(\alpha,\beta))=(2,1)$.
By the analogous above discussion, we see that
for each $(i,j)\in \{(1,2),(1,0),(0,1)\}$ 
there exists $(\alpha'',\beta'')\in \Delta_K$
such that $\varepsilon_{K}(\alpha'',\beta'')=(2,1)$ and
$\varepsilon_{K}(T_{K}(\alpha'',\beta''))=(i,j)$.
By applying permutations of the coordinates of $(\gamma,\alpha,\beta)$
we get Theorem \ref{t3} for the totally real cubic field $K$.
We consider the case where $K$ has complex embeddings.
Let $\mu$ be the real root of $x^3-5$.
Put $K_2={\Bbb Q}(\mu)$.
The direct calculation implies   Table \ref{table:B}.

\begin{table}[hbtp]
\caption{(not totally real case) The words of length $2$ which are 
not forbidden by Table 
\ref{table:forbidden}
eventually occur.}
\label{table:B}
\begin{center}
\input{table3.tex}
\end{center}
\end{table}

Using Table \ref{table:B}, we can show Theorem \ref{t3} 
for the case where $K$ is not a totally real cubic field. 
\end{proof}

For the periodic continued fraction obtained by this algorithm, we have the
following   Proposition \ref{P1}, which can be shown by using Theorem \ref{th:primitive}  
in a way similar to Perron \cite{p}. 
We denote  by $\Delta_{K}^{\mathcal{P}er}=\Delta_{K,r}^{\mathcal{P}er}$ the set of the periodic points of the 
transformation $T=T_{K,r}$, i.e.,
\begin{equation}
\label{al:per}
\Delta_{K}^{\mathcal{P}er}=
\Delta_{K,r}^{\mathcal{P}er}:=
\left\{
\left( \alpha,\beta \right) \in \Delta_K \left|
\begin{array}{l}
\mbox{ there exist }
m,n \in {\Bbb Z}_{>0} \\
\mbox{ such that }
m \neq n 
\mbox{ and }\\
T^{m}_{K,r}
\left(
\alpha,\beta
\right)=T^{n}_{K,r}
\left(
\alpha,\beta \right)
\end{array}
\right.
\right\}.
\end{equation}

\begin{proposition}\label{P1}
Let $(\alpha,\beta)\in\Delta_{K}^{\mathcal{P}er}$.
Then, there exists a constant $c(\alpha,\beta)>0$ and $\eta(\alpha,\beta)>0$ 
such that $\eta(\alpha,\beta)\leq \frac{3}{2}$  and
\[
\begin{array}{lcl}
\left| \alpha-\dfrac{q_n}{p_n} \right|\leq \dfrac{c(\alpha,\beta)}{p_n^{\eta(\alpha,\beta)}}, & \quad &
\left| \beta-\dfrac{r_n}{p_n} \right|\leq \dfrac{c(\alpha,\beta)}{p_n^{\eta(\alpha,\beta)}},\\[10pt]
\left| \alpha-\dfrac{q'_n}{p'_n}\right|\leq \dfrac{c(\alpha,\beta)}{(p'_n)^{\eta(\alpha,\beta)}}, &\quad &
\left| \beta-\dfrac{r'_n}{p'_n}\right|\leq \dfrac{c(\alpha,\beta)}{(p'_n)^{\eta(\alpha,\beta)}},\\[10pt]
\left| \alpha-\dfrac{q''_n}{p''_n}\right|\leq \dfrac{c(\alpha,\beta)}{(p''_n)^{\eta(\alpha,\beta)}},&\quad &
\left| \beta-\dfrac{r''_n}{p''_n}\right|\leq \dfrac{c(\alpha,\beta)}{(p''_n)^{\eta(\alpha,\beta)}}
\end{array}
\]
hold.
Furthermore, $\eta(\alpha,\beta)= \frac{3}{2}$ holds
if and only if $K$ is not a totally real cubic field.
\end{proposition}

We can also give some examples of periodic  expansions 
in the similar manner as in \cite{TY}.

\begin{theorem}\label{P2}
Let $m\in {\Bbb Z}_{>0}$.
Let $\lambda$ be the real root of
$x^3-mx^2-1$. Let $K={\Bbb Q}(\lambda)$.
Then, 
\[
\left( \frac{1}{1+\lambda+\lambda^2},\frac{\lambda}{1+\lambda+\lambda^2}
\right)
\in {\mathcal{P}er}_K^{5/2}.
\]
\end{theorem}


\begin{proof}
For $n\in {\Bbb Z}_{\geq 0}$ let $(\alpha_n,\beta_n)=T_{K,5/2}^{n}
(\frac{1}{1+\lambda+\lambda^2},\frac{\lambda}{1+\lambda+\lambda^2})
$.
Then, we will prove  for $0\leq k \leq m-1$
\begin{align*}
&\alpha_k=\frac{\lambda}{(m-k+1)\lambda^2+\lambda+1},\ 
\beta_k=\frac{\lambda^2}{(m-k+1)\lambda^2+\lambda+1}\\
&\varepsilon \left(\alpha_k,\beta_k \right)=(0,2),\\
&\alpha_{m+k}=\frac{\lambda^2}{(m-k+1)\lambda^2+\lambda+1},\ 
\beta_{m+k}=\frac{(m-k)\lambda^2+1}{(m-k+1)\lambda^2+\lambda+1}\\ 
&\varepsilon \left(\alpha_{m+k},\beta_{m+k} \right)=(2,1),\\
&\alpha_{2m+k}=\frac{(m-k)\lambda^2+1}{(m-k+1)\lambda^2+\lambda+1},\ 
\beta_{2m+k}=\frac{\lambda}{(m-k+1)\lambda^2+\lambda+1}\\ 
&\varepsilon \left(\alpha_{2m+k},\beta_{2m+k} \right)=(1,0).\\
\end{align*}
In what follows, we suppose  that $k\in {\Bbb Z}$.
First, let $\zeta_k,\eta_k$ and $\xi_k$ be 
\begin{align*}
&\zeta_k=\frac{\lambda}{(m-k+1)\lambda^2+\lambda+1},\ 
\eta_k=\frac{\lambda^2}{(m-k+1)\lambda^2+\lambda+1},\\ 
&\xi_{k}=\frac{(m-k)\lambda^2+1}{(m-k+1)\lambda^2+\lambda+1}.\\
\end{align*}
Simple calculations show the following:
\begin{enumerate}
\item[(1)]
$\zeta_k+\eta_k+\xi_k=1$,
\item[(2)]
$\zeta_0=\alpha_0$ and $\eta_0=\beta_0$,
\item[(3)]
$\left(\zeta_k,\eta_k \right)\in \Delta_K$ holds, 
\item[(4)]
$T_{(0,2)} \left(\zeta_k,\eta_k \right)=
\left(\zeta_{k+1},\eta_{k+1} \right)$ holds,
\item[(5)]
$\left(\zeta_{m},\eta_{m} \right)=
\left(\eta_{0},\xi_{0} \right)$ holds.
\end{enumerate}
We prove that  $\varepsilon(\zeta_k,\eta_k)=(0,2)$ for 
$0 \leq k<m$. 
It is easy to see that
\begin{align*}
&N(\zeta_k)=\frac{1}{m^2+(k^2-3k)m-k^3+3k^2},\\
&N(\eta_k)=\frac{1}{m^2+(k^2-3k)m-k^3+3k^2},\\
&N(\xi_k)=\frac{k^2m-k^3+1}{m^2+(k^2-3k)m-k^3+3k^2}.
\end{align*}
Therefore, we have
\begin{align*}
\frac{\zeta_k^{5/2}}{N(\zeta_k)}
&=\frac{\lambda^{5/2}(m^2+(k^2-3k)m-k^3+3k^2)}{((m-k+1)\lambda^2+\lambda+1)^{5/2}},\\
\frac{\eta_k^{5/2}}{N(\eta_k)}
&=\frac{\lambda^5(m^2+(k^2-3k)m-k^3+3k^2)}{((m-k+1)\lambda^2+\lambda+1)^{5/2}},\\
\frac{\xi_k^{5/2}}{N(\xi_k)}
&=\frac{((m-k)\lambda^2+1)^{5/2}(m^2+(k^2-3k)m-k^3+3k^2)}{(k^2m-k^3+1)((m-k+1)\lambda^2+\lambda+1)^{5/2}}.
\end{align*}
Since $\lambda>1$, we have $\frac{\zeta_k^{5/2}}{N(\zeta_k)}<\frac{\eta_k^{5/2}}{N(\eta_k)}$. 
We easily see that $m^2(m-k)^2\geq k^2m-k^3+1$.
Since $m<\lambda<m+1$, we see that
\begin{align*}
\left( \left(m-k \right) \lambda+\frac{1}{\lambda} \right)^{5/2}
>m^2(m-k)^2 \geq k^2m-k^3+1,
\end{align*}
which implies 
\begin{align*}
\frac{\zeta_k^{5/2}}{N(\zeta_k)}<\frac{\xi_k^{5/2}}{N(\xi_k)}.
\end{align*}
Thus, we have 
$\varepsilon(\zeta_k,\eta_k)=(0,2)$.
We can easily prove that $(\alpha_k,\beta_k)=(\zeta_k,\eta_k)$ for $0\leq k \leq m-1$,
  $(\alpha_{m+k},\beta_{m+k})=(\eta_k,\xi_k)$
and $(\alpha_{2m+k},\beta_{2m+k})=(\xi_k,\zeta_k)$
 on induction of $k$.
\end{proof}

\section{Continued Fraction Expansion and Acceleration of Continued Fraction}
\label{sec:cont}

As we have already seen that for any given 
$(x,y) \in \Delta_{K}$ for any given real cubic field $K$, 
we can consider a sequence $\left\{ \varepsilon \left( \alpha_{n}, \beta_{n} \right) \right\}_{n=0}^{\infty}$ defined by
\[
\varepsilon \left( \alpha_{n}, \beta_{n} \right) := \varepsilon \left( T^{n} (\alpha,\beta) \right) \in Ind:=\left\{
\left( i, j \right) \left|~
i, j \in \left\{ 0, 1, 2 \right\}, i \neq j
\right.
\right\}
\]
obtained by the algorithm given in Section \ref{sec:algo}. 

In this section, we shall describe the continued fraction 
expansion of 
$\frac{1}{1-\alpha-\beta} (\alpha,\beta)$
according to the "expansion" 
$\left\{ \varepsilon \left( \alpha_{n}, \beta_{n} \right) \right\}_{n=0}^{\infty}$ of 
$(\alpha,\beta)$.  In what follows of this section, 
we use column vectors instead of row vectors.  
We denote by 
$\bar{\bar{x}}= ~^{t} (x^{(0)}, x^{(1)}, x^{(2)}) \in \mathbb{R}^{3}$
(resp., $\bar{x}=~^{t} (x^{(1)}, x^{(2)}) \in \mathbb{R}^{2}$)
an vector of dimension 3 (resp., of dimension 2), where $t$
indicates the transpose. 

We need some definitions. 
For $n\in {\Bbb Z}_{>0}$ and a set $S$
we denote by $M(n,S)$ $n\times n$ matrices with 
entries in $S$.
We put 
\begin{align}
\label{eq:C}
C \left( \bar{a} \right)  &:=
\begin{pmatrix}
{}^{t} \bar{0} & 1\\ 
E_{2} &  \bar{a}
\end{pmatrix}, 
\
 \bar{a} \in \mathbb{Z}^{2}_{\geq 0},\\
P_{n}  =  \left( \bar{\bar{p}}_{n-2}\ 
\bar{\bar{p}}_{n-1} \
\bar{\bar{p}}_{n}
\right)
 & :=  
 C \left( \bar{a}_{0} \right) 
 C \left( \bar{a}_{1} \right) 
 \cdots
 C \left( \bar{a}_{n} \right),
 \
  \bar{a}_{n} \in \mathbb{Z}^{2}_{\geq 0}, \nonumber
\end{align}
\[ 
\left( n \geq -1, \ P_{-1}:=E_{3}, \ 
 \bar{\bar{p}}_{n}  =
   {}^{t}\left( p_{n}^{\left( 0 \right)}, p_{n}^{\left( 1 \right)}, 
 p_{n}^{\left( 2 \right)}
 \right)
 \right),
\]
where $E_{m}$ is the unit matrix of size $m \times m$.

We write
\begin{align*}
\frac{1}{{x \choose y}} &:= {1/y \choose x/y}
\quad \left( x, y \in \mathbb{R}, \ y \neq 0 \right),\\
\left[ \bar{a}_{0}; \bar{a}_{1}, \ldots, \bar{a}_{n} \right]
&:=
\bar{a}_{0} + \dfrac{1}{
\bar{a}_{1}+\dfrac{1}{
\ddots\begin{array}{c}
~\\
+\dfrac{1}{\bar{a}_{n}}
\end{array}}},
\end{align*}
and 
\[
\left[
\bar{a}_{0}; \bar{a}_{1}, \bar{a}_{2}, \ldots
\right]:=
\lim_{n \rightarrow \infty} 
\left[
\bar{a}_{0}; \bar{a}_{1}, \bar{a}_{2}, \ldots, \bar{a}_n
\right]
\]
as far as the limit exists.

Using Lemma \ref{lem:AB} and 
\[
C \left( \bar{a} \right)^{\mbox{\it frac}}
{x \choose y} = \pi \kappa 
\begin{pmatrix}
y \\
1+ay\\
x + by
\end{pmatrix}
= {a \choose b} + \dfrac{1}{{x \choose y}},
\]
where $\bar{a}=
\begin{pmatrix}
a \\
b
\end{pmatrix}$,
we get the following formula, see for example, 
\cite{NS, T2}.

\begin{formula}

\begin{enumerate}
\item[(1)]
\(
\left[
\bar{a}_{0}; \bar{a}_{1},  \ldots, \bar{a}_{n}, \bar{x}
\right]=P_{n}^{\mbox{\it frac}} \left( \bar{x} \right);
\)
\item[(2)]
\(\displaystyle
\left[
\bar{a}_{0}; \bar{a}_{1}, \ldots, \bar{a}_{n}
\right]=\dfrac{1}{p_{n}^{\left( 0 \right)}}
{p_{n}^{\left( 1 \right)} \choose p_{n}^{\left( 2 \right)}}
\) 
holds provided $p_{n}^{\left( 0 \right)}
\neq 0.$
\end{enumerate}
\end{formula}

It is convenient to write two dimensional continued fraction
\(
\left[
\bar{a}_{0}; \bar{a}_{1},  \ldots, \bar{a}_{n}
\right]
\)
(resp., 
\(
\left[
\bar{a}_{0}; \bar{a}_{1}, \bar{a}_{2}, \ldots
\right]
\)) as a finite word
(resp., an infinite word) over 
$\mathbb{Z}^{2}_{\geq 0}$:
\[
\begin{array}{rcl}
\mbox{(CF) }  
\bar{a}_{0} \bar{a}_{1}  \ldots \bar{a}_{n}
& := & 
\left[
\bar{a}_{0}; \bar{a}_{1}, \ldots, \bar{a}_{n}
\right],\\
\mbox{(CF) }  
\bar{a}_{0} \bar{a}_{1}  \bar{a}_{2} \ldots
& := & 
\left[
\bar{a}_{0}; \bar{a}_{1}, \bar{a}_{2}, \ldots
\right].
\end{array}
\]

For $\bar{x}={}^{t} \left( x, y \right) \in \triangle$, we 
denote by $\bar{x}^{*} = {}^{t} \left( x^{*}, y^{*} \right)$
a vector defined by
\[
\bar{x}^{*} = {x^{*} \choose y^{*}} = \dfrac{1}{1-x-y} {x \choose y}.
\]
Notice that 
${}^{t}\left( x^{*}, y^{*} \right) \in \mathbb{R}^{2}_{>0}$
always holds for any $\bar{x}={}^{t}\left( x, y \right) \in \triangle$. 

Now we can state our theorem.

\begin{theorem}
\label{th:varepsilon}
Let $\left\{ \varepsilon_n \right\}_{n=0}^{\infty}$=$\left\{ \varepsilon \left( \alpha_{n}, \beta_{n} \right) \right\}^{\infty}_{n=0}$
be the expansion of $(\alpha,\beta)$
$\in \Delta_{K}$.  Then
\begin{equation}
\label{eq:CFW}
{\alpha^{*} \choose \beta^{*}}
= \mbox{ (CF) } W_{\varepsilon_{0}} W_{\varepsilon_{1}}
\ldots W_{\varepsilon_{n-1}} \ldots
\end{equation}
where
\begin{align*}
W_{ \left( 1,2 \right)} & = 
{0 \choose 0} {0 \choose 0} {0 \choose 1}, 
&W_{\left( 0,1 \right)} & = 
{0 \choose 0} {0 \choose 1} {0 \choose 0}, 
&W_{\left( 0,2 \right)} & = 
{0 \choose 0} {0 \choose 0} {1 \choose 0},\\
W_{\left( 2,1 \right)} & =  
{0 \choose 0} {1 \choose 0} {0 \choose 0},
&W_{\left( 1,0 \right)} & = 
{1 \choose 0} {0 \choose 0} {0 \choose 0}, 
&W_{\left( 2,0 \right)} & = 
{0 \choose 1} {0 \choose 0} {0 \choose 0}.
\end{align*}
\end{theorem}

One can check the following lemmas by direct calculation.

\begin{lemma}
\label{lem:astar}
Let $A^{*}_{\left( i, j \right)} = S^{-1} A_{\left( i, j \right)}
S$ ($\left( i, j \right) \in Ind$), where $A_{\left( i, j \right)}$
and $S$ are matrices as in Section \ref{sec:algo}.  Then, 
\[
A^{*}_{\left( i, j \right)} = M_{\left( i, j \right)} 
\quad \left( \forall \left( i, j \right) \in Ind \right),
\] 
where $M_{\left( i, j \right)} = \left(
m_{k \ell} \right)_{0 \leq k \leq 2, 0 \leq \ell \leq 2} \in GL_{3} \left( \mathbb{Z} \right)$ defined by 
\[
m_{k \ell} := \left\{
\begin{array}{rcl}
1 & \quad k=\ell \mbox{ or } \left( k, \ell \right) = \left( i, j \right) \\
0 & \quad \mbox{ otherwise}
\end{array}.
\right.
\]
\end{lemma}

\begin{lemma}
\label{lem:RUV}
Let $R$, $U$, $V$ be matrices defined by 
\[
R=C \left( \bar{0} \right), \quad
U=C \left( \bar{u} \right), \quad
V=C \left( \bar{v} \right),
\]
where 
$\bar{0} = {}^{t} \left( 0, 0 \right)$, 
$\bar{u} = {}^{t} \left( 1, 0 \right)$, 
$\bar{v} = {}^{t} \left( 0, 1 \right)$,
and $C \left( \bar{a} \right)$ ( $\bar{a} \in \mathbb{Z}_{\geq0}^{2}$)
is the matrix (\ref{eq:C}).  Then, 
\begin{align*}
M_{\left( 1,2 \right)} & =  RRV,  
&M_{\left( 0,1 \right)} & =  RVR,
&M_{\left( 0,2 \right)} & =  RRU,\\
M_{\left( 2,1 \right)} & =  RUR, 
&M_{\left( 1,0 \right)}  &=  URR,  
&M_{\left( 2,0 \right)} & =  VRR.
\end{align*}
\end{lemma}

\begin{proof}[Proof of Theorem \ref{th:varepsilon}] \
One can see that $\varepsilon_{0} 
= 
\left( i, j \right) \in Ind$ implies
\[
A_{\left( i, j \right)}^{-1} 
\begin{pmatrix}
1 \\
\alpha \\
\beta
\end{pmatrix}
\sim
\begin{pmatrix}
1 \\
\alpha_{1} \\
\beta_{1}
\end{pmatrix} .
\]
Hence, by induction, we have
\[
\begin{pmatrix}
1 \\
\alpha \\
\beta
\end{pmatrix}
\sim
A_{\varepsilon_{0}} A_{\varepsilon_{1}} \cdots A_{\varepsilon_{n-1}} 
\begin{pmatrix}
1 \\
\alpha_{n} \\
\beta_{n}
\end{pmatrix}, \
(\varepsilon_{k}=\varepsilon \left( \alpha_{k}, \beta_{k} \right)), 
\]
so that 
\[
S^{-1}
\begin{pmatrix}
1 \\
\alpha \\
\beta
\end{pmatrix}
\sim
S^{-1} A_{\varepsilon_{0}} S S^{-1} 
A_{\varepsilon_{1}} S \cdots S^{-1} A_{\varepsilon_{n-1}} S 
S^{-1} 
\begin{pmatrix}
1 \\
\alpha_{0} \\
\beta_{0}
\end{pmatrix}
=
A_{\varepsilon_{0}}^{*} 
A_{\varepsilon_{1}}^{*} 
\cdots
A_{\varepsilon_{n-1}}^{*} 
S^{-1}
\begin{pmatrix}
1 \\
\alpha_{n} \\
\beta_{n}
\end{pmatrix} ,
\]
which implies
\[
{\alpha^{*} \choose \beta^{*}} = 
A_{\varepsilon_{0}}^{* \mbox{\it frac}}
A_{\varepsilon_{1}}^{* \mbox{\it frac}}
\cdots
A_{\varepsilon_{n-1}}^{* \mbox{\it frac}}
{\alpha_{n}^{*} \choose \beta_{n}^{*}},  
\]
\[
{\alpha^{*} \choose \beta^{*}} := \dfrac{1}{1-\alpha-\beta} {\alpha \choose \beta}, 
\quad
{\alpha_{n}^{*} \choose \beta_{n}^{*}} := \dfrac{1}{1-\alpha_{n}-\beta_{n}} {\alpha_{n} \choose \beta_{n}}
\]
which together with Lemma \ref{lem:astar} and
Lemma \ref{lem:RUV}, we get Theorem  \ref{th:varepsilon}.
\end{proof}

We can make reduction of the continued fraction of the form
(\ref{eq:CFW}) in Theorem \ref{th:varepsilon}
by applying the following reduction rule: 
\[
\mbox{ (CF) } \cdots {a \choose b}
{0 \choose 0}
{0 \choose 0}
{c \choose d}
\cdots = 
\mbox{ (CF) } 
\cdots 
{a+c \choose b+d}
\cdots, 
\]
in particular
\[
\mbox{ (CF) } \cdots {a \choose b}
{0 \choose 0}
{0 \choose 0}
{0 \choose 0}
{c \choose d}
\cdots = 
\mbox{ (CF) } 
\cdots 
{a \choose b}
{c \choose d}
\cdots.
\]

We give an example.
Let $\lambda$ be the real root of $x^{3} -mx^{2}
-1$, ($m \in \mathbb{Z}_{>0}$)
as in Theorem \ref{P2}.
Then, Theorem \ref{P2} says that
\begin{equation}
\label{eq:gzi}
\dfrac{1}{1-\xi-\eta} {\xi \choose \eta} = 
\mbox{(CF)} 
W_{\left(0, 2 \right)}^{m}
W_{\left(2,1 \right)}^{m}
W_{\left(1,0 \right)}^{m}
W_{\left(0,2 \right)}^{m}
W_{\left(2,1 \right)}^{m}
W_{\left(1,0 \right)}^{m}
\cdots,
\end{equation}
\[
\xi = \frac{1}{1+\lambda+\lambda^2}, \quad
\eta=\frac{\lambda}{1+\lambda+\lambda^2}.
\]
Hence, applying the reduction rule repeatedly, we get 
periodic continued fractions: 
\begin{eqnarray}
\label{eq:continued}
\dfrac{1}{1-\xi-\eta} {\xi \choose \eta} & = & 
\stackrel{~}{\mbox{(CF)}}
\stackrel{*}{{0 \choose 0}} 
\stackrel{~}{{0 \choose 0}} 
\stackrel{~}{{m \choose 0}} 
\stackrel{~}{{0 \choose 0}} 
\stackrel{~}{{m \choose 0}} 
\stackrel{~}{{0 \choose 0}} 
\stackrel{~}{{m \choose 0}} 
\stackrel{~}{{0 \choose 0}} 
\stackrel{*}{{0 \choose 0}} \nonumber\\
& = & \stackrel{~}{\mbox{(CF)}}
\stackrel{~}{{0 \choose 0}} 
\stackrel{*}{{0 \choose 0}} 
\stackrel{~}{{m \choose 0}} 
\stackrel{~}{{0 \choose 0}} 
\stackrel{~}{{m \choose 0}} 
\stackrel{~}{{0 \choose 0}} 
\stackrel{*}{{m \choose 0}} \nonumber\\
& = & \stackrel{~}{\mbox{(CF)}}
\stackrel{~}{{0 \choose 0}} 
\stackrel{*}{{0 \choose 0}} 
\stackrel{*}{{m \choose 0}},
\end{eqnarray}
which is an accelerated continued fraction of 
(\ref{eq:gzi}).  In other words, if we say 
(\ref{eq:continued}) is a canonical continued fraction,
then the expression 
(\ref{eq:CFW}) given in Theorem \ref{th:varepsilon} can be
considered as a (slow or additive) continued fraction expansion of a canonical continued fraction. 
For $n\in {\Bbb Z}_{>0}$ $M \in M(n,{\Bbb Z}_{\geq 0})$  
is called primitive, if there exists an positive integer $m$ such that 
$M^m\in M(n,{\Bbb Z}_{> 0})$.

\begin{lemma}
\label{lem:primitive}
Let $
\left\{
\varepsilon_n 
\right\}^{\infty}_{n=0}$$=\left\{ \varepsilon \left( \alpha_{n}, \beta_{n} \right) \right\}^{\infty}_{n=0}$
be the expansion of $(\alpha,\beta) \in \Delta_{K}$. 
We suppose that $\left\{  
\varepsilon_n \right\}^{\infty}_{n=0}$
is purely periodic and the length of the period is $l$.
Then, every eigenvalue of the matrix $M_{\varepsilon_{0}}\cdots M_{\varepsilon_{l-1}}$ 
is in $K\setminus{\Bbb Q}$.
\end{lemma}
\begin{proof}
We suppose that an eigenvalue $\lambda$ of 
the matrix $M_{\varepsilon_{0}}\cdots M_{\varepsilon_{l-1}}$ denoted by $C=(c_{i,j})$
associated with the eigenvector ${}^t(1-\alpha-\beta,\alpha,\beta)$
is a rational number.
Then, we have  
\begin{align}\label{Cmat}
C
\begin{pmatrix}
1-\alpha-\beta\\
\alpha\\
\beta
\end{pmatrix}
=
\lambda
\begin{pmatrix}
1-\alpha-\beta\\
\alpha\\
\beta
\end{pmatrix}.
\end{align}
Since $1-\alpha-\beta,\alpha,\beta $ are linearly independent over ${\Bbb Q}$,
$C=\lambda E$, where $E$ is the unit matrix.
Let $(i',j')=\varepsilon_{0}$. Then, $C=M_{\varepsilon_{0}}\cdots M_{\varepsilon_{l-1}}$ 
implies that  $c_{i',j'}\geq 1$, which leads to the contradiction that
$C=\lambda E$.
\end{proof}

\begin{remark}
The irreducibility of the matrix $M_{\varepsilon_{0}}\cdots M_{\varepsilon_{l-1}}$
 follows from Lemma \ref{lem:primitive} as far as we are concentrated with
cubic field $K$.
\end{remark}

\begin{theorem}
\label{th:primitive}
Let $
\left\{
\varepsilon_n 
\right\}^{\infty}_{n=0}=\left\{ \varepsilon \left( \alpha_{n}, \beta_{n} \right) \right\}^{\infty}_{n=0}$
be the expansion of $(\alpha,\beta) \in \Delta_{K}$. 
We suppose that $\left\{
\varepsilon_n \right\}^{\infty}_{n=0}$
is purely periodic and the length of the period is $l$.
Then, the matrix $M_{\varepsilon_{0}}\cdots M_{\varepsilon_{l-1}}$ 
is primitive.
\end{theorem}

\begin{proof}
For each $n\in {\Bbb Z_{\geq 0}}$
we put $C_n$ by  
\begin{align*}
C_n=(\,_nc_{i,j})=
M_{\varepsilon_{0}}\cdots 
M_{\varepsilon_{n}}.
\end{align*}
We see easily that $\,_nc_{i,j}\leq \,_{n+1}c_{i,j}$
 for each $n\in {\Bbb Z_{\geq 0}}$ 
and $i,j$ with $0\leq i,j \leq 2$.
We suppose that there exists $i_0,j_0$ with $0\leq i_0,j_0 \leq 2$ such that 
$\lim_{n\to \infty}\,_nc_{i_0,j_0}=0$, which is equivalent to that
for every $n\in {\Bbb Z}_{\geq 0}\  \,_nc_{i_0,j_0}=0$. 
Since we see $\,_0c_{i,i}=1$ for every $i$ with $0\leq i \leq 2$, 
we have $i_0\ne j_0$. 
For simplicity, we consider the case where $(i_0,j_0)=(0,2)$.
 First, we suppose that 
   $\lim_{n\to \infty}\,_nc_{0,1}=0$, which is equivalent to that
for every $n\in {\Bbb Z}_{\geq 0}\  \,_nc_{0,1}=0$.
Since $|C_n|$=1 for every $n\in {\Bbb Z}_{\geq 0}$, it follows that 
$\,_nc_{0,0}=1$ for every $n\in {\Bbb Z}_{\geq 0}$.
Then, we have 
\begin{align*}
{}^{t}C_{l-1}
\begin{pmatrix}
1 \\
0 \\
0
\end{pmatrix}
=
\begin{pmatrix}
1 \\
0 \\
0
\end{pmatrix},
\end{align*}
which implies that $1$ is an eigenvalue of $M_{\varepsilon_{0}}\cdots M_{\varepsilon_{l-1}}$, which contradicts Lemma \ref{lem:primitive}.
Next, we suppose that 
   $\lim_{n\to \infty}\,_nc_{0,1}>0$, which is equivalent to that
there exists  $n_0\in {\Bbb Z}_{\geq 0}$
such that $\,_{n_0}c_{0,1}>0$. 
We see that if $n>n_0$, then  
$\varepsilon_n $
 $\notin \{(0,2),(1,2)\}$.
Therefore, by the pure periodicity of 
$\left\{ 
\varepsilon_n
\right\}_{\geq 0}$ 
we get  $
\varepsilon_n \notin \{(0,2),(1,2)\}$ for all $n\in {\Bbb Z}_{\geq 0}$.
Hence, we can write 
\begin{align*}
C_{l-1}=
\left(
\begin{array}{ccc}
*&*&0 \\
*&*&0 \\
*&*&1
\end{array}
\right),
\end{align*}
so that 
\begin{align*}
C_{l-1}
\begin{pmatrix}
0 \\
0 \\
1
\end{pmatrix}
=
\begin{pmatrix}
0 \\
0 \\
1
\end{pmatrix},
\end{align*}
which contradicts Lemma \ref{lem:primitive}.
Thus, we have proved the theorem for the case where 
$(i_0,j_0)=(0,2)$.
We can do the same for the other cases.
\end{proof}

\section{Numerical experiments}
\label{sec:num}
We put 
\begin{align*}
\mathrm{dh}
\left(
\frac{p}{q} 
\right):=\max\{\lfloor \log_{10}|p|+1 \rfloor,\lfloor \log_{10}|q|+1 \rfloor \}, \ \ \ \mathrm{dh}(0):=0
\end{align*}
for $\frac{p}{q}$ ($p,q\in {\Bbb Z}$ are coprime).
The function $\mathrm{dh}$ can be extended to ${\Bbb Q}[x]$:
\begin{align*}
\mathrm{dh}(g):=\max_{0\leq i \leq n}\{\mathrm{dh}(a_i)\},
\end{align*}
for $\displaystyle{g(x)=\sum_{i=0}^n a_ix^i\in {\Bbb Q}[x]}$.
We define $\overline{\mathrm{dh}}$
, $\mathrm{dh_{F}}$
 and $\mathrm{rdh_{F}}$
by 
\begin{align*}
&\overline{\mathrm{dh}}({\bf \alpha}):=\max_{i\in\{1,2\}} 
\{\mathrm{dh}(\phi_{\alpha_i}) \},\\
&\mathrm{dh_{F}}(n;{\bf \alpha}):=\overline{\mathrm{dh}}(T^{n}_K({\bf \alpha})),\\
&\mathrm{rdh_{F}}(n;{\bf \alpha}):=\frac{\overline{\mathrm{dh}}(T^{n}_K({\bf \alpha}))}
{\overline{\mathrm{dh}}({\bf \alpha})},\\
&\mathrm{dh_{F}}({\bf \alpha}):=\max_{n\in {\Bbb Z}_{\geq 0}}\{\mathrm{dh_{F}}(n;{\bf \alpha})\},\\
&\mathrm{rdh_{F}}({\bf \alpha}):=\max_{n\in {\Bbb Z}_{\geq 0}}\{\mathrm{rdh_{F}}(n;{\bf \alpha})\},
\end{align*}
for ${\bf \alpha}=(\alpha_1,\alpha_{2})\in \Delta_K$ and 
$n\in {\Bbb Z}_{\geq 0}$, 
where $\phi_{\alpha_i}\in {\Bbb Q}[x]\  (i\in \{1,2\})$
is the monic minimal polynomial of $\alpha_i$. 
The function $\mathrm{dh_{F}}(n;{\bf \alpha})$
(resp., $\mathrm{rdh_{F}}(n;{\bf \alpha})$)
is referred to as {\it the $n$th decimal height of ${\bf \alpha}$}
(resp., {\it the $n$th relative decimal height of ${\bf \alpha}$})
.
We computed the length of the periods of 
$\left( \langle \sqrt[3]{m}\rangle/2,
\langle\sqrt[3]{m^2}\rangle/2 \right)$ for $T_{K,r}$ with $K={\Bbb Q}(\sqrt[3]{m}), r=5/2$  
for all $m\in {\Bbb Z}$ with 
$2\leq m\leq 10000$ $\left(\sqrt[3]{m}\notin {\Bbb Q}\right)$ 
and these decimal heights, cf. Table \ref{table:C} given below, 
where $\langle x\rangle$ is the fractional part of $x$.
For the calculation of the tables, we used
 a computer  equipped with GiNaC \cite{g} on GNU C++\footnote{The routine that was written for this purpose can be downloaded from the web site 
http://www.lab2.toho-u.ac.jp/sci/c/math/yasutomi/mfarey.html}.
We confirmed that $\left(\langle \sqrt[3]{m}\rangle/2,
\langle \sqrt[3]{m^2}\rangle/2 \right)\in \Delta_{K, 5/2}^{\mathcal{P}er}$ for all $m$ with $2\leq m\leq 10000 \left(
\sqrt[3]{m}\notin {\Bbb Q} \right)$.

\begin{center}
\begin{longtable}{c|r|r|r}
\caption{The result of periodicity test for 
$\left( \langle \sqrt[3]{m}\rangle/2,
\langle\sqrt[3]{m^2}\rangle/2 \right)$ for 
all noncubic positive integers $2 \leq m \leq 10000$ with 
$r=5/2$.}
\label{table:C}\\
\hline
\multicolumn{1}{c}{\text{Range of $m$ ($m_1\leq m\leq m_2$)}} &
\multicolumn{1}{c}{\text{$L_A(m_1,m_2)$}} 
& 
\multicolumn{1}{c}{\text{$H_A(m_1,m_2)$}} 
& \multicolumn{1}{c}{\text{$R_A(m_1,m_2)$}} 
\\
 \hline 
\endfirsthead

\multicolumn{4}{c}%
{{\tablename\ \thetable{} -- continued from previous page}} \\

\hline 
\multicolumn{1}{c}{\text{Range of $m$ 
($m_1\leq m\leq m_2$)}} &
\multicolumn{1}{c}{\text{$L_A(m_1,m_2)$}} &
\multicolumn{1}{c}{\text{$H_A(m_1,m_2)$}} 
& \multicolumn{1}{c}{\text{$R_A(m_1,m_2)$}} 
\\ \hline 
\endhead

\multicolumn{4}{r}{
\small\sl Continued on next page}\\
\hline
\endfoot

\hline
\endlastfoot

$2\leq m \leq 200$ &4494 &7&3\\
\hline
$201 \leq m \leq 400$&13641&8&7/3
\\
\hline
$401 \leq m \leq 600$ &13578&8&2
\\
\hline
$601 \leq m \leq 800$ &30447&8&2
\\
\hline
$801 \leq m \leq 1000$ &36963&8&2
\\
\hline
$1001\leq m \leq 1200$ &31119&9&9/4\\
\hline
$1201 \leq m \leq 1400$ &68529&9&9/5\\
\hline
$1401 \leq m \leq 1600$ &65310&9&9/5\\
\hline
$1601 \leq m \leq 1800$ &62598&9&9/5\\
\hline
$1801 \leq m \leq 2000$ &52551&9&9/5\\
\hline
$2001\leq m \leq 2200$ &74931&10&2\\
\hline
$2201 \leq m \leq 2400$ &177570&9&9/5\\
\hline
$2401 \leq m \leq 2600$ &97446&9&9/5\\
\hline
$2601 \leq m \leq 2800$ &79923&9&9/5\\
\hline
$2801 \leq m \leq 3000$ &121134&10&9/5\\
\hline
$3001\leq m \leq 3200$&107577&9&9/5\\
\hline
$3201 \leq m \leq 3400$&107919&10&2\\
\hline
$3401 \leq m \leq 3600$&95388&10&9/5\\
\hline
$3601 \leq m \leq 3800$&150393&10&9/5\\
\hline
$3801 \leq m \leq 4000$&133650&10&9/5\\
\hline

$4001 \leq m \leq 4200$&137787&10&2\\
\hline

$4201 \leq m \leq 4400$& 242391&10&2\\
\hline

$4401 \leq m \leq 4600$& 322374 &10& 2\\
\hline

$4601 \leq m \leq 4800$&180246 &10 & 2\\
\hline

$4801 \leq m \leq 5000$&124335&10&2\\
\hline

$5001\leq m \leq 5200$ &282870&10&2\\
\hline
$5201 \leq m \leq 5400$ &169845&10&2\\
\hline
$5401 \leq m \leq 5600$ &134589&10&2\\
\hline
$5601 \leq m \leq 5800$ &236004&10&2\\
\hline
$5801 \leq m \leq 6000$ &298266&10&2\\
\hline
$6001\leq m \leq 6200$ &439470&10&2\\
\hline
$6201 \leq m \leq 6400$ &249141&10&2\\
\hline
$6401 \leq m \leq 6600$ &188673&10&2\\
\hline
$6601 \leq m \leq 6800$ &176733&10&2\\
\hline
$6801 \leq m \leq 7000$ &462093&11&2\\
\hline
$7001\leq m \leq 7200$ &160650&10&5/3\\
\hline
$7201 \leq m \leq 7400$ &619809&10&5/3\\
%
%
\hline
$7401 \leq m \leq 7600$ &241893&10&5/3\\
\hline
$7601 \leq m \leq 7800$ &254790&10&5/3\\
\hline
$7801 \leq m \leq 8000$ &232170&10&5/3\\
\hline
$8001\leq m \leq 8200$ &398433&11&11/6\\
\hline
$8201 \leq m \leq 8400$&211460&11&11/6
\\
\hline
$8401 \leq m \leq 8600$ &264786&10&5/3
\\
\hline
$8601 \leq m \leq 8800$ &293934&11&11/6
\\
\hline
$8801 \leq m \leq 9000$ &785715&10&5/3
\\
\hline
$9001\leq m \leq 9200$ &265377&11&11/6\\
\hline
$9201 \leq m \leq 9400$ &377157&11&11/6\\
\hline
$9401 \leq m \leq 9600$ &258939&10&5/3\\
\hline
$9601 \leq m \leq 9800$ &269877&10&5/3\\
\hline
$9801 \leq m \leq 10000$ &276768&10&5/3

\end{longtable}
\end{center}

%

In Table \ref{table:C}, $L_A(m_1,m_2)$, $H_A(m_1,m_2)$ and $R_A(m_1,m_2)$ are numbers
 defined by
\begin{align*}
&L_A(m_1,m_2):=\text{the maximum value of the length of the shortest
period of}\\
&\text{ the expansion of\ }
(\langle \sqrt[3]{m} \rangle/2,\langle \sqrt[3]{m^2}\rangle/2)\ 
\text{for\ }
m_1\leq m \leq m_2\ \text{with}\ \sqrt[3]{m}\notin {\Bbb Q},\\ 
&H_A(m_1,m_2):=\max_{m_1\leq m \leq m_2,\sqrt[3]{m}\notin {\Bbb Q}}\mathrm{dh_{F}}(\langle \sqrt[3]{m} \rangle/2,\langle \sqrt[3]{m^2}\rangle/2),
\\
&R_A(m_1,m_2):=\max_{m_1\leq m \leq m_2, \sqrt[3]{m}\notin {\Bbb Q},}
\mathrm{rdh_{F}}(\langle \sqrt[3]{m} \rangle/2,\langle \sqrt[3]{m^2}\rangle/2),\end{align*}
which are well-defined by the periodicity.
This Table \ref{table:C} together with following numerical experiments by PCs for some totally real cubic fields
etc. support Conjecture \ref{con1} given at the end of 
our paper, which says that $\Delta_K= \Delta_{K, r}^{\mathcal{P}er}$ holds, cf. (\ref{al:per}).
On the other hand, the $explosion\  phenomenon$ takes place  
if we apply classical algorithms (the Jacobi-Perron algorithm, the modified  
Jacobi-Perron algorithm etc.), cf. \cite{TY, TY3, TY4}.

Let $K$ be a real cubic field 
 and let $\alpha^{(0)}, \alpha^{(1)}, \alpha^{(2)}$
be its positive ${\Bbb Q}$-basis with
\[
\alpha=\frac{\alpha^{(1)}}{\alpha^{(0)}+\alpha^{(1)}+\alpha^{(2)}},\
\beta=\frac{\alpha^{(2)}}{\alpha^{(0)}+\alpha^{(1)}+\alpha^{(2)}}. 
\]
Let  $\left \{ 
\varepsilon_n \right\}_{n=0}^{\infty}$
be the expansion of the $(\alpha,\beta)$.
Suppose that $\varepsilon_{k},\ldots, \varepsilon_{k+l-1}$ 
is the period of the expansion.
Then
$(1-\alpha_{k}-\beta_{k},\alpha_{k},\beta_{k})$
becomes an eigenvector with respect to an eigenvlaue of 
$M_{\varepsilon_{k}} \cdots 
M_{\varepsilon_{k+l-1}}$ 
which will be denoted by $\lambda \left( \alpha, \beta 
\right)$.
The eigenvalue is important for the Diophantine approximation to 
$(\alpha,\beta)$. 
We denote by $N_{t}$ a set 
\[
\left\{
\left. 
\left(\frac{1}{1+\alpha+\alpha^2},\frac{\alpha}{1+\alpha+\alpha^2}\right) \right|~
\begin{array}{l}
\alpha \mbox{ is the positive maximal root of }\\
\mbox{some irreducible } p\in P_{t} 
\end{array}
\right\},
\]
\[
P_{t}:=\{x^3+a_2x^2+a_1x+a_0
\left| a_i\in {\Bbb Z},\ |a_i|\leq t\ \text{for\ }i=0,1,2
\right.
\} \ \left( t >0 \right).
\]
We put $n_t$, $p_t$, $c_t$, $r_t$, $s_t$, $rh_t$ as follows:
\begin{eqnarray*}
n_t &= &\sharp N_{t},\\
p_t & = & \sharp \left\{ \left(\alpha,\beta \right)\in N_{t} \left|~
\left( \alpha,\beta \right)  \mbox{ is periodic by our algorithm}
\right.
\right\},\\
c_t & = & \sharp \left\{
\left(
\alpha,\beta
\right)
\in N_{t}
\left|
~{\Bbb Q}
\left(
\alpha,\beta
\right)
\mbox{ has a complex embedding}
\right.
\right\}\\
r_t & = & \sharp
\left\{
\left(
\alpha,\beta
\right)
\in N_{t}
\left|~
{\Bbb Q}
\left(
\alpha,\beta
\right)
\mbox{ is a 
totally real cubic field}
\right.
\right\},\\  
s_t  & = & \sharp 
\left\{
\left(
\alpha,\beta
\right)
\in N_{t}
\left|~ 
\lambda
\left(
\alpha,\beta
\right)
\mbox{ is the Pisot number}
\right.
\right\}\\
rh_t & = & 
\max \left\{
{\rm rdh}_{F}
\left(
\alpha,\beta
\right)
\left|~
\left(
\alpha,\beta
\right)
\in N_{t} 
\right.
\right\}.
\end{eqnarray*}
Then, we get $n_{15}=18797$, $p_{15}=18797$, $c_{15}=7689$, $r_{15}=11108$, $s_{15}=18797$ and
$rh_{15}=7/3$, i.e., every  $(\alpha,\beta)\in N_{15}$ is periodic by this algorithm, and 
 $\lambda(\alpha,\beta)$
 becomes Pisot number for 
 all the element $(\alpha,\beta)\in N_{15}$
without any exceptions (cf. Conjecture \ref{con2} given in Section \ref{sec:con}
).
For this calculation, we used
 a computer  equipped with GiNaC \cite{g} on GNU C++
 \footnote{The routine that was written for this purpose can be downloaded from the web site 
http://www.lab2.toho-u.ac.jp/sci/c/math/yasutomi/mfarey.html}.
We note that the maximal eigenvalue of 
$\prod_{\lambda\in \Lambda, M_\lambda\in BM}M_\lambda$ ($\Lambda < \infty$) is not 
always a Pisot number, where
 $BM:=\{M_{(i,j)} \left|~0\leq i,j\leq 2, i\ne j \right.\}$.
For example, $M_{(1,0)}M_{(0,1)}M_{(2,0)}M_{(0,2)}$ has a non-Pisot maximal eigenvalue.

\section{Stepped surfaces and Substitutions}
\label{sec:stepped}

We shortly prepare the geometric tools.  
Let us denote 
by $\bar{\bar e}_{i}$ $\left( i=0,1,2 \right)$
the canonical basis 
 of ${\mathbb R}^{3}$, i.e., 
\[
\bar{\bar e}_{0} :=~^{t} \left( 1, 0, 0 \right),\
\bar{\bar e}_{1} :=~^{t} \left( 0, 1, 0 \right),\
\bar{\bar e}_{2} :=~^{t} \left( 0, 0, 1 \right).
\]
For $\bar{\bar x} \in \mathbb{Z}^{3}$, $i=0,1,2$,
we mean by 
$\left( \bar{\bar x}, i^{*} \right)$ 
a unit square defined by 
\[
\left( \bar{\bar x}, i^{*} \right) := 
\left\{
\bar{\bar x} + t \bar{ \bar e}_{j} + u \bar{\bar e}_{k} 
~\left|~ t, u \in \left[ 0, 1 \right], \ 
\left\{ i, j , k \right\} = \left\{ 0,1,2 \right\} \right.
\right\}
\]
(see Figure \ref{fig:chips}).

\begin{figure*}[hbtp]
\begin{center}
\begin{minipage}{3cm}
\begin{center}
\includegraphics[width=3cm]{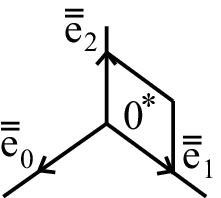}\\
$\left( \bar{\bar 0}, 0^{*} \right)$
\end{center}
\end{minipage}
\quad
\begin{minipage}{3cm}
\begin{center}
\includegraphics[width=3cm]{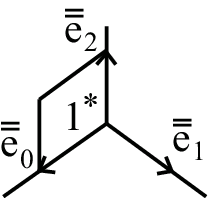}\\
$\left( \bar{\bar 0}, 1^{*} \right)$
\end{center}
\end{minipage}
\quad
\begin{minipage}{3cm}
\begin{center}
\includegraphics[width=3cm]{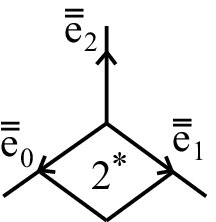}\\
$\left( \bar{\bar 0}, 2^{*} \right)$
\end{center}
\end{minipage}
\quad
\begin{minipage}{3cm}
\begin{center}
\includegraphics[width=3cm]{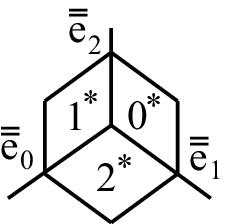}\\
$\sum_{i=0}^{2} \left( \bar{\bar 0}, i^{*} \right)$
\end{center}
\end{minipage}
\end{center}
\caption{$\left( \bar{\bar 0},i^{*} \right)$ and $\sum_{i=0}^{2} \left( \bar{\bar 0}, i^{*} \right)$.}
 \label{fig:chips}
\end{figure*}

Let $\bar{\bar \alpha} = {}^{t} \left( 
\alpha^{\left( 0 \right)}, 
\alpha^{\left( 1 \right)}, 
\alpha^{\left( 2 \right)}
\right) \in \mathbb{R}^{3}_{>0}$ and
$\alpha^{\left( 0 \right)}$, 
$\alpha^{\left( 1 \right)}$,
$\alpha^{\left( 2 \right)}$
be linearly independent over $\mathbb{Q}$.  
Notice that without loss of generality, we may assume 
$\alpha^{\left( 0 \right)}+ \alpha^{\left( 1 \right)}+
\alpha^{\left( 2 \right)}=1$
.
We consider the sets
\(
 {\mathcal P} \left( \bar{ \bar \alpha } \right),
\)
\(
 {\mathcal P}^{>} \left( \bar{ \bar \alpha } \right),
\)
and
\(
 {\mathcal P}^{\geq} \left( \bar{ \bar \alpha } \right):
\)
\begin{eqnarray*}
 {\mathcal P} \left( \bar{ \bar \alpha } \right)
 &:=&
\left\{
\bar{ \bar x} \in \mathbb{R}^{3} \left|~ \left< 
\bar{\bar x}, \bar{\bar \alpha} \right>=0 \right.
\right\},\\
 {\mathcal P}^{>} \left( \bar{ \bar \alpha } \right)
 &:=&
\left\{
\bar{ \bar x} \in \mathbb{R}^{3} \left|~ \left< 
\bar{\bar x}, \bar{\bar \alpha} \right>>0 \right.
\right\},\\
 {\mathcal P}^{\geq} \left( \bar{ \bar \alpha } \right)
 &:=&
\left\{
\bar{ \bar x} \in \mathbb{R}^{3} \left|~ \left< 
\bar{\bar x}, \bar{\bar \alpha} \right>\geq 0 \right.
\right\}.
\end{eqnarray*}
where $\left< \cdot, \cdot \right>$ means the inner product.
We put 
%
%
\begin{eqnarray*}
{\mathscr S} \left( \bar{\bar \alpha} \right)
&:=&
\left\{
\left( \bar{\bar x}, i^{*} \right) \left|~ i=0,1,2, \
\left< \bar{\bar x}, \bar{\bar \alpha} \right> > 0, \
\left<\bar{\bar x}-\bar{\bar e}_{i} ,  \bar{\bar \alpha}  \right> \leq 0 
\right.
\right\}, \\
 {\mathscr S}' \left( \bar{\bar \alpha} \right)
&:=&
\left\{
\left( \bar{\bar x}, i^{*} \right) \left|~ i=0,1,2, \
\left< \bar{\bar x}, \bar{\bar \alpha} \right> \geq 0, \
\left<\bar{\bar x}-\bar{\bar e}_{i} ,  \bar{\bar \alpha}  \right> < 0 
\right.
\right\},
\end{eqnarray*}
which are  subsets
of 
${\mathbb Z}^{3} \times
\left\{ 0^{*}, 1^{*}, 2^{*} \right\}$.
By the definition of 
\(
\mathscr{S} \left( \bar{\bar \alpha} \right)
\)
(and
\(
\mathscr{S}' \left( \bar{\bar \alpha} \right)
\)),
we see that
\(
\mathscr{S} \left( \bar{\bar \alpha} \right)
\)
(and
\(
\mathscr{S}' \left( \bar{\bar \alpha} \right)
\))
consists of
the nearest
unit squares clinging to the plane
\(
 {\mathcal P}\left( \bar{ \bar \alpha } \right).
\)
%
\(
\mathscr{S} \left( \bar{\bar \alpha} \right)
\)
(and
\(
\mathscr{S}' \left( \bar{\bar \alpha} \right)
\))
will be referred to as the 
 {\it stepped surface} 
 with respect to the direction 
 $\bar{\bar \alpha}$.  We also says that 
\(
\mathscr{S} \left( \bar{\bar \alpha} \right)
\)
(and
\(
\mathscr{S}' \left( \bar{\bar \alpha} \right)
\))
is the 
stepped surface of the plane 
$ {\mathcal P} \left( \bar{\bar \alpha} \right)$. 
\begin{remark}
The difference between 
${\mathscr S} \left( \bar{\bar \alpha} \right)$
and ${\mathscr S}' \left(  \bar{\bar \alpha} \right)$
is 
that
\(
\mathscr{S} \left( \bar{\bar \alpha} \right)
\backslash
\mathscr{S}' \left( \bar{\bar \alpha} \right)
=
\left\{
\left( \bar{\bar e}_i, i^{*} \right)
\right\}_{i=0,1,2},
\)
\(
\mathscr{S}' \left( \bar{\bar \alpha} \right)
\backslash
\mathscr{S} \left( \bar{\bar \alpha} \right)
=
\left\{
\left( \bar{\bar 0}, i^{*}  \right)
\right\}_{i=0,1,2}.
\)
\end{remark}

Moreover, we define 
${\mathcal  S} \left( \bar{\bar \alpha} \right)
$: 
\[
{\mathcal  S} \left( \bar{\bar \alpha} \right) := \left\{ \Lambda
\left|~
\# \Lambda < + \infty, \ \Lambda \subset {\mathscr  S}  
\left( \bar{\bar \alpha} \right)\right.
\right\}.
\]
We denote by ${\mathcal  G}  \left( \bar{\bar \alpha} \right)$ the $\mathbb{Z}$-free module 
generated by all the finite squares:
\begin{equation}
\label{eq:G}
{\mathcal  G} \left( \bar{\bar \alpha} \right) := 
\left\{ \left. \sum_{\left( 
\bar{\bar x}, i^{*} 
\right): m_{\left( 
\bar{\bar x}, i^{*} 
\right)}\neq 0
}
\hspace{-5mm} m_{\left( \bar{\bar x}, i^{*} \right)}
\left( \bar{\bar x}, i^{*} \right) 
\right|~
\begin{array}{l}
\bar{\bar x} \in \mathbb{Z}^{3}, \
i \in \left\{ 0,1,2 \right\}, \\
m_{\left( \bar{\bar x},i^{*} \right)} \in \mathbb{Z}, \
\left( \bar{\bar x}, i^{*} \right) 
\in {\mathscr S} \left( \bar{\bar \alpha} \right),\\
\# 
\left\{
 \left( \bar{\bar x}, i^{*} \right) 
\left|~ m_{\left( \bar{\bar x}, i^{*} \right)} \neq 0 \right.
\right\} < +\infty 
\end{array}
\right\}.
\end{equation}

Remark that the sum on the right-hand side of 
(\ref{eq:G})
is a formal sum.  Hence, 
we have 
\(
{\mathcal  G}\left( \bar{\bar \alpha} \right)  = \left\{ \sum_{\lambda \in \Lambda}
 m_{\lambda} \lambda
 \left|~
 \Lambda \subset {\mathscr  S} \left( \bar{\bar \alpha} \right) , \
 \#  \Lambda < + \infty, \ m_{\lambda} \in 
 \mathbb{Z}
\right.
\right\}.
\)
In what follows, we only consider the case $m_{\lambda}=1$.
When $m_{\lambda}=1$, we call the element 
$\sum_{\lambda \in \Lambda} \lambda $ of 
${\mathcal  G} \left( \bar{\bar \alpha} \right) $ a {\it patch} of 
${\mathscr S}\left( \bar{\bar \alpha} \right) $.
In some cases, 
\(
\left\{\lambda \right\}_{\lambda \in \Lambda}
\)
of 
${\mathcal  S} \left( \bar{\bar \alpha} \right)$
is also called as a {\it patch}, but 
an 
element of 
${\mathcal  S} \left( \bar{\bar \alpha} \right)$
and 
${\mathcal  G} \left( \bar{\bar \alpha} \right)$
should be distinguished. 
It will be convenient to
define two maps, 
${}_s \Psi_g: {\mathcal  S} 
\left( \bar{\bar \alpha} \right)
\rightarrow {\mathcal  G}\left( \bar{\bar \alpha} \right)$
and
${}_g \Psi_s: {\mathcal  G} \left( \bar{\bar \alpha} \right)
\rightarrow {\mathcal  S} \left( \bar{\bar \alpha} \right)$ 
as follows: for $\Lambda \in {\mathcal S} \left( \bar{\bar \alpha} \right)$
\begin{eqnarray*}
 {}_s \Psi_g \left(\left\{
\lambda
\right\}_{{\lambda \in \Lambda} }\right)
&:=&
\sum_{\lambda \in \Lambda}  
\lambda \in {\mathcal  G} \left( \bar{\bar \alpha} \right)
\quad
\mbox{ (for $\left\{ \lambda  \right\}_{\lambda \in \Lambda} \in 
{\mathcal S} \left( \bar{\bar \alpha} \right)$),}\\
 {}_g \Psi_s \left(
\sum_{\lambda \in \Lambda}  
\lambda 
\right)
&:=&
\left\{
\lambda
\right\}_{{\lambda \in \Lambda} }
\in 
{\mathcal S} \left( \bar{\bar \alpha} \right)
\quad
\mbox{ (for $\sum_{\lambda \in \Lambda}  
\lambda 
 \in 
{\mathcal G} \left( \bar{\bar \alpha} \right)$).}
\end{eqnarray*}
For example, for $\sum_{i=0}^{2} \left( \hako{e}_{i},
i^{*} \right)$
and
$\left\{ \left( \hako{e}_{i},
i^{*} \right) \right\}_{i=0,1,2}$,
$ {}_g \Psi_s \left( \sum_{i=0}^{2} \left( \hako{e}_{i},
i^{*} \right)\right) =$ \\
$\left\{ \left( \hako{e}_{i},
i^{*} \right) \right\}_{i=0,1,2}$,
$ {}_s \Psi_g
\left(
\left\{ \left( \hako{e}_{i}, i^{*} \right)  \right\}_{i=0,1,2}
\right)
=\sum_{i=0}^{2} \left( \hako{e}_{i},
i^{*} \right)$.
For $\gamma$, $\delta \in {\mathcal G} 
\left( \bar{\bar \alpha} \right)$,
 we denote $\gamma \prec \delta$ if 
 $ {}_g \Psi_s \left( \gamma \right) \subset 
{}_g \Psi_s \left( \delta \right)$.
%
%
%
%
%
%
%
%
%
%
%
Taking 
${\mathscr S}' \left( \bar{\bar \alpha} \right)$ 
instead of  
${\mathscr S} \left( \bar{\bar \alpha} \right)$,
we can  define 
${\mathcal S}' \left( \bar{\bar \alpha} \right)$ 
and
${\mathcal G}' \left( \bar{\bar \alpha} \right)$ 
similarly.
For each $\left( i, j \right) \in Ind$, 
we consider
the substitution
$\sigma_{\left( i, j \right)}$ as follows:
\[
\sigma_{\left( i, j \right)}:\left\{
\begin{array}{rcl}
  i & \mapsto & ji \\
  k & \mapsto & k  \quad \left(  k \neq i 
\right)
\end{array}
\right..
\]
And 
the so-called 
incidence matrix
$L_{\left( i, j \right)}$
of $\sigma_{\left( i, j \right)}$
is the square matrix of size $3 \times 3$ defined by
$L_{\left( i, j \right)}= \left( l_{i'j'} \right)$ where
$l_{i'j'}$ is the number of occurrences of a 
letter $i'$ appearing  
in $\sigma_{\left( i, j \right)} \left( j' \right)$.
%
Notice that
$L_{\left( i, j \right)} =M_{\left( j, i \right)}$
for each 
$\left( i, j  \right) \in Ind$
where $M_{ \left(j,i \right)}$ 
is the matrix given in Lemma \ref{lem:astar}.
For
\( 
\left( 
\alpha, \beta 
\right)
\in \Delta_{K}
\), 
we put
\(
\bar{\bar \nu} \left( \alpha, \beta \right) :=
{}^{t}
\left( 1-\alpha-\beta, \alpha, \beta\right).
\)
In the sequel we fix an arbitrary  
\(
\left( 
\alpha, \beta 
\right)
\in \Delta_{K}
\)
and an arbitrary
$n\in {\mathbb Z}_{\geq 0}$.
Then, 
the dual substitution $\Theta_{\varepsilon_n}$ of $\sigma_{\varepsilon_n}$, which is an endomorphism
from 
${\mathcal  G} \left( \bar{\bar \nu} \left( \alpha_{n+1}, \beta_{n+1} 
\right) \right)$ 
to
${\mathcal  G} \left( \bar{\bar \nu} \left( \alpha_{n}, \beta_{n} 
\right) \right)$ 
introduced in \cite{AI}, can be defined by
\begin{align}
\label{eq:Thetadef}
&\Theta_{\varepsilon_n} \left( \bar{\bar x}, i^{*} \right)  := 
L_{\varepsilon_n}^{-1} 
\bar{\bar x}  + 
\sum_{j=0}^{2} \sum_{\scriptsize \begin{array}{c}
S:\\
\sigma_{\varepsilon_n} \left( j \right) = 
P i  S
\end{array}} \left( L_{\varepsilon_n}^{-1} \left( f \left( S \right) \right),j^{*} \right), \\
& \hspace{2cm} \Theta_{\varepsilon_n} \left( \sum_{\lambda \in \Lambda} \left( \bar{\bar x}, i^{*} \right)_{\lambda} \right) 
 :=  \sum_{\lambda \in \Lambda} \left( \Theta_{\varepsilon_n} \left( \bar{\bar x}, i^{*} \right)_{\lambda} \right)\nonumber
\end{align}
for $i=0,1,2$,
where 
$f \left( w \right) :=
\)
\( {}^{t} \left( 
\left| w \right|_{0}, 
\left| w \right|_{1}, 
\left| w \right|_{2}
\right)$
($\left| w \right|_{i}$ is the number of 
occurrences of a symbol $i$ appearing in a finite word $w \in \left\{ 0,1, 2 \right\}^{*}$), and $P$ (resp., $S$) 
means that the prefix
(resp., suffix) of $i$ of $\sigma \left( j \right)$
with $\sigma \left( j \right) = 
P  i  S$ and 
\[
\bar{\bar y}+
\sum_{\lambda \in \Lambda} \left( \bar{\bar x}_\lambda, i^{*}_\lambda \right)
:=\sum_{\lambda \in \Lambda} \left(\bar{\bar y}+\bar{\bar x}_\lambda, i^{*}_\lambda \right).
\]
In particular, 
\begin{equation}
\label{eq:dualSubst}
\Theta_{\left( i,j \right)}:
\left\{
\begin{array}{rcl}
\left( \bar{\bar x}, j^{*} \right) & \mapsto & 
\left( 
L_{\left(i,j \right)}^{-1}
\left( \bar{\bar x} + \bar{\bar e}_{i} \right), i^{*} 
 \right)
+
\left( L_{\left( i,j \right)}^{-1}
\bar{\bar x}, j^{*} \right)
,\\
\left( \bar{\bar x}, k^{*} \right) & \mapsto & 
\left( L_{ \left( i,j \right)}^{-1}
\bar{\bar x}, k^{*} \right)
\quad \left( k \neq j \right).
\end{array}
\right.
\end{equation}
(see Figure \ref{fig:Theta}).
\begin{figure*}[hbtp]
\begin{center}
\begin{minipage}{1cm}
\begin{center}
\includegraphics[width=1cm]{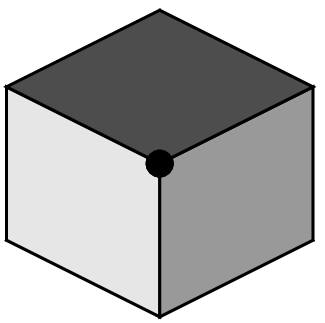}
\end{center}
\end{minipage}
\begin{minipage}{1cm}
\begin{center}
$\stackrel{\Theta_{\left( 1,2 \right)}}{\longrightarrow}$
\end{center}
\end{minipage}
\begin{minipage}{1.5cm}
\begin{center}
\includegraphics[width=1cm]{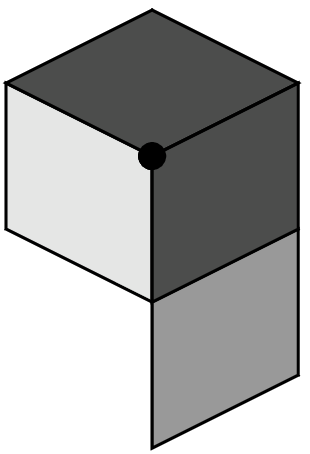}
\end{center}
\end{minipage}
 \quad \quad
\begin{minipage}{1cm}
\begin{center}
\includegraphics[width=1cm]{figure4-l.eps}
\end{center}
\end{minipage}
\begin{minipage}{1cm}
\begin{center}
$\stackrel{\Theta_{\left( 2,1 \right)}}{\longrightarrow}$
\end{center}
\end{minipage}
\begin{minipage}{1.5cm}
\begin{center}
\includegraphics[width=1.5cm]{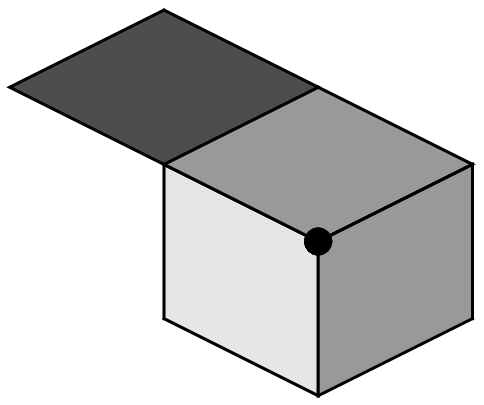}
\end{center}
\end{minipage}\\

\begin{minipage}{1cm}
\begin{center}
\includegraphics[width=1cm]{figure4-l.eps}
\end{center}
\end{minipage}
\begin{minipage}{1cm}
\begin{center}
$\stackrel{\Theta_{\left( 0,1 \right)}}{\longrightarrow}$
\end{center}
\end{minipage}
\begin{minipage}{1.5cm}
\begin{center}
\includegraphics[width=1.5cm]{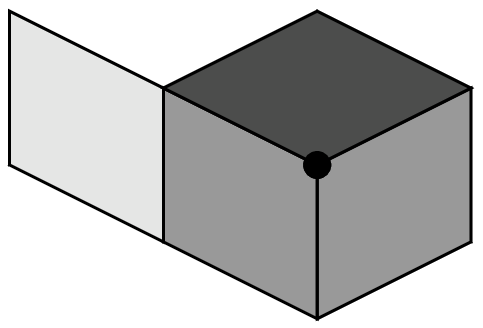}
\end{center}
\end{minipage}
\quad \quad
\begin{minipage}{1cm}
\begin{center}
\includegraphics[width=1cm]{figure4-l.eps}
\end{center}
\end{minipage}
\begin{minipage}{1cm}
\begin{center}
$\stackrel{\Theta_{\left( 1,0 \right)}}{\longrightarrow}$
\end{center}
\end{minipage}
\begin{minipage}{1.5cm}
\begin{center}
\includegraphics[width=1.5cm]{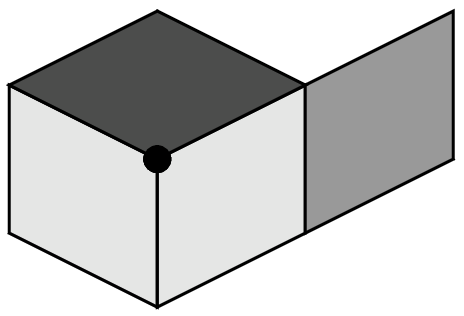}
\end{center}
\end{minipage}\\

\begin{minipage}{1cm}
\begin{center}
\includegraphics[width=1cm]{figure4-l.eps}
\end{center}
\end{minipage}
\begin{minipage}{1cm}
\begin{center}
$\stackrel{\Theta_{\left( 0,2 \right)}}{\longrightarrow}$
\end{center}
\end{minipage}
\begin{minipage}{1.5cm}
\begin{center}
\includegraphics[width=1cm]{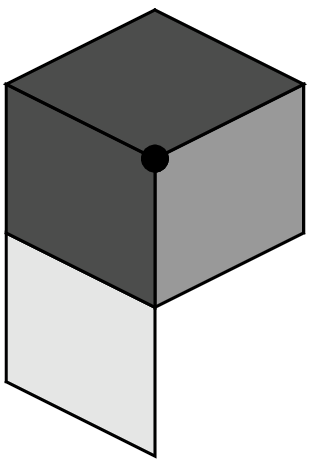}
\end{center}
\end{minipage}
\quad \quad
\begin{minipage}{1cm}
\begin{center}
\includegraphics[width=1cm]{figure4-l.eps}
\end{center}
\end{minipage}
\begin{minipage}{1cm}
\begin{center}
$\stackrel{\Theta_{\left( 2,0 \right)}}{\longrightarrow}$
\end{center}
\end{minipage}
\begin{minipage}{1.5cm}
\begin{center}
\includegraphics[width=1.5cm]{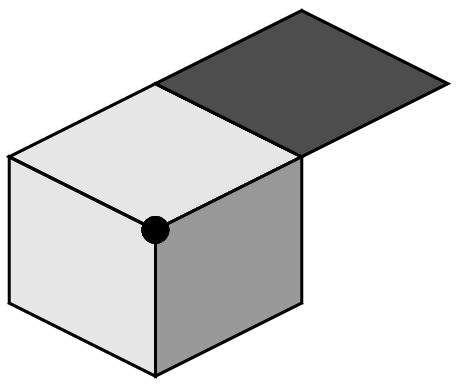}
\end{center}
\end{minipage}\\
\end{center}
\caption{
$\Theta_{\varepsilon} \left( {\mathcal U} \right)$,
${\mathcal U}=\sum_{i=0}^{2} \left( \hako{e}_{i}, 
i^{*} \right)$,
 $\varepsilon \in Ind$.}
\label{fig:Theta}
\end{figure*}

In view of definitions, we have
\begin{lemma}
\label{lem:u}
There exists a constant $c_{n}$ satisfying
\[
\bar{\bar \nu} \left( \alpha_{n+1}, \beta_{n+1} \right)= c_{n} {}^{t}  L_{\varepsilon_n}^{-1}
\bar{\bar \nu} \left( \alpha_{n}, \beta_{n} \right)
\]
where 
\[
c_{n} := 
\left\{
\begin{array}{lcl}
\frac{1}{1-\beta_{n}} & \mbox{ if } &
\varepsilon_n
\in 
\left\{ 
\left( 1,2 \right), 
\left( 0,2 \right)
\right\},
\medskip\\
\frac{1}{1-\alpha_{n}} & \mbox{ if } &
\varepsilon_n
\in 
\left\{
\left( 2,1 \right),
\left( 0,1 \right)
\right\},
\medskip\\
\frac{1}{\alpha_{n}+\beta_{n}} & \mbox{ if } &
\varepsilon_n
\in 
\left\{
\left( 1,0 \right),
\left( 2,0 \right)
\right\}.
\end{array}
\right.
\]
\end{lemma}

\begin{proof}
Let 
$\varepsilon_n= \left( 1,2 \right)$, then from the fact that 
\[
\left( \alpha_{n+1}, \beta_{n+1} \right)=
T_{\left( 1,2 \right)} 
\left( \alpha_{n}, \beta_{n} \right) = 
\left( 
\dfrac{\alpha_{n}-\beta_{n}}{1-\beta_{n}}, 
\dfrac{\beta_{n}}{1-\beta_{n}}
\right),
\]
we have
\begin{eqnarray*}
\bar{\bar \nu} \left( \alpha_{n+1}, \beta_{n+1} \right)& = & 
\dfrac{1}{1-\beta_{n}} 
{}^{t} \left( 1-\alpha_{n}-\beta_{n}, \ \alpha_{n}-\beta_{n}, \ \beta_{n} \right) \\
& =&  
\dfrac{1}{1-\beta_{n}} 
\begin{pmatrix}
1 & 0 & 0\\
0 & 1 & -1\\
0 & 0 & 1
\end{pmatrix}
\begin{pmatrix}
1-\alpha_{n}-\beta_{n}\\
\alpha_{n} \\
\beta_{n} 
\end{pmatrix}
=
c_{n} 
{}^{t} L_{\left( 1,2 \right)}^{-1}
\bar{\bar \nu} \left( \alpha_{n}, \beta_{n} \right).
\end{eqnarray*}
The proof for the other cases is analogously done.
\end{proof}

We define
$\varphi_{n}:\mathbb{R}^{3} \rightarrow
\mathbb{R}^{3}$ by
\[
\bar{\bar x}_{n}= 
{}^{t} \left( x_{n}, y_{n}, z_{n} \right)=
\varphi_{n} \left( \bar{\bar x}_{n+1}  \right)= \varphi_{n} 
{}^{t} \left( x_{n+1}, y_{n+1}, z_{n+1} \right)
:=
L_{\varepsilon_n}^{-1} \  {}^{t} \left( x_{n+1}, y_{n+1}, z_{n+1} \right).
\]
Then, we have
\begin{lemma}
\label{lem:varphi}
Let $c_{n}$ be numbers
given in Lemma \ref{lem:u}.  Then,
\[
\left< \varphi_{n} \bar{\bar x}_{
n+1}, 
\bar{\bar \nu} \left( \alpha_{n}, \beta_{n} \right)\right>
=
\dfrac{1}{c_{n}} 
\left< \bar{\bar x}_{
n+1}, \bar{\bar \nu} \left( \alpha_{n+1}, \beta_{n+1} \right)\right>, \quad \bar{\bar x}_{n+1} \in \mathbb{R}^{3}
\]
holds.
\end{lemma}

\begin{proof}
By Lemma \ref{lem:u}, we get
\begin{eqnarray*}
\left< \varphi_{n} \bar{\bar x}_{n+1}, \bar{\bar \nu} \left( \alpha_{n}, \beta_{n} \right)\right>
&=&\left<  L_{\varepsilon_n}^{-1} \bar{\bar x}_{n+1},
\bar{\bar \nu} \left( \alpha_{n}, \beta_{n} \right)\right>
=\left<\bar{\bar x}_{n+1}, {}^{t}L_{\varepsilon_n}^{-1} 
\bar{\bar \nu} \left( \alpha_{n}, \beta_{n} \right)\right>\\
&=&\dfrac{1}{c_{n}} 
\left<\bar{\bar x}_{n+1},
\bar{\bar \nu} \left( \alpha_{n+1}, \beta_{n+1} \right)\right>.
\end{eqnarray*}
\end{proof}
From Lemma \ref{lem:varphi}, 
it follows
\begin{corollary}
We have
\begin{eqnarray*}
\varphi_{n} \left( 
{\mathcal P}^{>} \left( 
\bar{\bar \nu} \left( \alpha_{n+1}, \beta_{n+1} \right)
\right)\right) &=&
{\mathcal P}^{>} \left( 
\bar{\bar \nu} \left( \alpha_{n}, \beta_{n} \right)
\right),\\
\varphi_{n} \left( 
{\mathcal P}^{\geq} \left( 
\bar{\bar \nu} \left( \alpha_{n+1}, \beta_{n+1} \right)
\right)\right) &=&
{\mathcal P}^{\geq} \left( 
\bar{\bar \nu} \left( \alpha_{n}, \beta_{n} \right)
\right).
\end{eqnarray*}
\end{corollary}
The following theorem is important 
related to the stepped surface.
\begin{theorem}['Bijectivity' of $\Theta$]
\label{th:tilingsubst}
Let us assume that 
$ \left( \alpha_{n}, \beta_{n} \right) \in 
\Delta_{K}$, then
\[
\Theta_{{\varepsilon_n} }: 
 {\mathcal G} \left( \bar{\bar \nu} \left( \alpha_{n+1}, \beta_{n+1} \right) \right)
 \rightarrow 
 {\mathcal  G}\left( \bar{\bar \nu} \left( \alpha_{n}, \beta_{n} \right) \right)
  \]
satisfies the following:
\begin{enumerate}
\item[(1)]
if $\left( \bar{\bar x}, i^{*} \right)$ is a unit square
of ${\mathcal G} \left( \bar{\bar \nu} \left( \alpha_{n+1}, \beta_{n+1} \right) \right)  $, then
the image $\Theta_{\varepsilon_n} \left( 
\bar{\bar x}, i^{*} \right)$ belongs to 
${\mathcal G} \left( \bar{\bar \nu} \left( \alpha_n, \beta_{n} \right) \right)$;
\item[(2)]
two distinct unit squares are sent to disjoint images (which are 
patches of squares) except for their boundaries;
\item[(3)]
for any $\left( \bar{\bar z}, k^{*} \right) \in
{\mathcal G} \left( \bar{\bar \nu} \left( \alpha_n, \beta_{n} \right) \right)$, there exists a square 
$\left( \bar{\bar x}, i^{*} \right) \in$\\
$ 
{\mathcal G} \left( \bar{\bar \nu} \left( \alpha_{n+1}, \beta_{n+1} \right) \right)$
such that 
$\Theta_{\varepsilon_n} \left( \bar{\bar x}, i^{*} \right) \succ \left( \bar{\bar z}, j^{*} \right)$.

\end{enumerate}
\end{theorem}

\begin{proof}
For instance, we consider the case where
$\varepsilon_n= \left( 2, 0 \right)$.  We shall
show the following three properties:
\begin{enumerate}
\item[(1)]
if $\left( \bar{ \bar x}, i^{*} \right) \in {\mathcal G} \left( \bar{\bar \nu} \left( \alpha_{n+1}, \beta_{n+1} \right) \right) 
$, then 
$\Theta_{\left( 2, 0 \right)} \left( \bar{ \bar x}, i^{*} \right)
\in {\mathcal  G} \left( \bar{\bar \nu} \left(\alpha_{n}, \beta_{n} \right) \right)$;
\item[(2)]
if $\left( \bar{ \bar x}, i^{*} \right) \neq \left( \bar{\bar y}, j^{*} \right)$ 
($\left( \bar{ \bar x}, i^{*} \right)$,
$\left( \bar{\bar y}, j^{*} \right)
\in {\mathcal G} \left( \bar{\bar \nu} \left( \alpha_{n+1}, \beta_{n+1} \right) \right) $), then \\
${}_g \Psi_s\left( \Theta_{\left( 2, 0 \right) }\left( \bar{ \bar x}, i^{*} \right)
\right)
\cap  {}_g \Psi_s\left( \Theta_{\left( 2, 0 \right) }\left( \bar{ \bar y}, j^{*} \right) \right)= \emptyset;$
\item[(3)]
if for any $\left( \bar{ \bar z}, k^{*} \right)
\in {\mathcal  G} \left( \bar{\bar \nu} \left(\alpha_{n}, \beta_{n} \right) \right)$, then 
there exists 
$\left( \bar{\bar x}, i^{*} \right) \in$\\
$
{\mathcal G} \left( \bar{\bar \nu} \left( \alpha_{n+1}, \beta_{n+1} \right) \right) $ such that
$\Theta_{\left( 2, 0 \right)} \left( \bar{\bar x}, i^{*} \right)
\succ \left( \bar{\bar z}, k^{*} \right)$.
\end{enumerate}

About (1): If $\left( \bar{\bar x}, i^{*} \right) \in
{\mathcal G} \left( \bar{\bar \nu} \left( \alpha_{n+1}, \beta_{n+1} \right) \right) $, $i=0,1,2$, then 
\begin{equation}
\label{eq:Theta21}
\begin{array}{rcl}
\Theta_{\left( 2, 0 \right)} \left( \bar{ \bar x}, 0^{*} \right)
& = & 
\left( L_{\left( 2,0 \right)}^{-1} \bar{\bar x}, 0^{*} \right)+
\left( L_{\left( 2,0 \right)}^{-1} \left( \bar{\bar x}  + \bar{\bar e}_{2} \right), 2^{*} \right), \\
\Theta_{\left( 2, 0 \right)} \left( \bar{ \bar x}, 1^{*} \right)
& = & 
\left(L_{\left( 2,0 \right)}^{-1} \bar{\bar x}, 1^{*} \right), 
 \\
\Theta_{\left( 2, 0 \right)} \left( \bar{ \bar x}, 2^{*} \right)
& = & 
\left(L_{\left( 2,0 \right)}^{-1} \bar{\bar x}, 2^{*} \right)
\end{array}
\end{equation}
follow from 
(\ref{eq:dualSubst}).
Hence, it suffices to show
\[
\left( L_{\left( 2,0 \right)}^{-1} \bar{\bar x}, 0^{*} \right), 
\left( L_{\left( 2,0 \right)}^{-1} \left( \bar{\bar x}  + \bar{\bar e}_{2} \right), 2^{*} \right),
\left(L_{\left( 2,0 \right)}^{-1} \bar{\bar x}, 1^{*} \right),
\left(L_{\left( 2,0 \right)}^{-1} \bar{\bar x}, 2^{*} \right)
\in {\mathcal  G} \left( \bar{\bar \nu} \left(\alpha_{n}, \beta_{n} \right) \right).
\]
From the definition of $\left( \bar{ \bar x}, i^{*} \right) \in
{\mathcal G} \left( \bar{\bar \nu} \left( \alpha_{n+1},
\beta_{n+1} \right) \right)$, we know that 
\begin{eqnarray}
\label{eqnarray:0up}
\left< \bar{\bar x}, \bar{ \bar \nu} \left( \alpha_{n+1}, \beta_{n+1} \right) \right> & > & 0, \\
\label{eqnarray:0down}
\left< \bar{\bar x}- \bar{ \bar e}_{i}, \bar{ \bar \nu} \left( \alpha_{n+1}, \beta_{n+1} \right) \right> & \leq & 0,
\end{eqnarray}
where the equality of (\ref{eqnarray:0down}) 
holds only for $\bar{\bar x}=\bar{\bar e}_{i}$.
We suppose that $\left(\bar{\bar x}, 0^{*} \right)\in {\mathcal G} \left( \bar{\bar \nu} \left( \alpha_{n+1}, \beta_{n+1} \right) \right)$.
By Lemma \ref{lem:u} and (\ref{eqnarray:0up}),
\[
\left<L_{\left( 2, 0 \right)}^{-1} \bar{ \bar x}, \bar{\bar \nu} \left( \alpha_{n}, \beta_{n} \right)\right> 
= 
\left< \bar{\bar x}, {}^{t}L_{\left( 2, 0 \right)}^{-1} \bar{\bar \nu} \left( \alpha_{n}, \beta_{n} \right)\right> =\dfrac{1}{c_{n}}
\left< 
\bar{\bar x}, 
\bar{\bar \nu} \left( \alpha_{n+1}, \beta_{n+1} \right)\right>
> 0,
\]
and 
by Lemma \ref{lem:u} and (\ref{eqnarray:0down}),
\(
\left<L_{\left( 2, 0 \right)}^{-1} \bar{ \bar x}-\bar{\bar e}_{0}, \bar{\bar \nu} \left( \alpha_{n}, \beta_{n} \right)\right>   = $\\
$
\left< L_{\left( 2, 0 \right)}^{-1} \left( 
\bar{\bar x} - \bar{\bar e}_{0} \right), \bar{\bar \nu} \left( \alpha_{n}, \beta_{n} \right)\right> 
=
\left<
\bar{\bar x} - \bar{\bar e}_{0}, {}^{t}  L_{\left( 2, 0 \right)}^{-1} 
\bar{\bar \nu} \left( \alpha_{n}, \beta_{n} \right)
\right>
=$\\
$\dfrac{1}{c_{n}}
\left<
\bar{\bar x} - \bar{\bar e}_{0}, 
\bar{\bar \nu} \left( \alpha_{n+1}, \beta_{n+1} \right)\right>\leq 0. 
\)
Hence, we obtain 
$\left( L_{\left( 2,0 \right)}^{-1} \bar{\bar x}, 0^{*} \right)$
$ \in {\mathcal  G} \left( \bar{\bar \nu} \left(\alpha_{n}, \beta_{n} \right) \right)$. 
By Lemma \ref{lem:u} and (\ref{eqnarray:0up}),
\(
\left<L_{\left( 2, 0 \right)}^{-1} 
\left( \bar{ \bar x}+ \bar{\bar e}_{2} \right), \bar{\bar \nu} \left( \alpha_{n}, \beta_{n} \right)\right> 
=\)\\
\( 
\left< \bar{\bar x}+ \bar{\bar e}_{2} , {}^{t}L_{\left( 2, 0 \right)}^{-1} \bar{\bar \nu} \left( \alpha_{n}, \beta_{n} \right)\right> =
\dfrac{1}{c_{n}}
\left< 
\bar{\bar x}+ \bar{\bar e}_{2}, 
\bar{\bar \nu} \left( \alpha_{n+1}, \beta_{n+1} \right)\right>
> 0
\)
and 
by Lemma \ref{lem:u} and (\ref{eqnarray:0down}),
\(
\left<L_{\left( 2, 0 \right)}^{-1} 
\left(
\bar{ \bar x}+\bar{\bar e}_{2}
\right)
-\bar{\bar e}_{2}, \bar{\bar \nu} \left( \alpha_{n}, \beta_{n} \right)\right>   
= 
\left< L_{\left( 2, 0 \right)}^{-1} \left( 
\bar{\bar x} - \bar{\bar e}_{0} \right), \bar{\bar \nu} \left( \alpha_{n}, \beta_{n} \right)\right> 
\)
\(=
\left<
\bar{\bar x} - \bar{\bar e}_{0}, {}^{t}  L_{\left( 2, 0 \right)}^{-1} 
\bar{\bar \nu} \left( \alpha_{n}, \beta_{n} \right)
\right>
\)
\(
=\dfrac{1}{c_{n}}
\left<
\bar{\bar x} - \bar{\bar e}_{0}, 
\bar{\bar \nu} \left( \alpha_{n+1}, \beta_{n+1} \right)\right>\leq 0. 
\)
Therefore, we get 
$\left( L_{\left( 2,0 \right)}^{-1}
\left(  \bar{\bar x}+ \bar{\bar e}_{2} \right)
, 2^{*} \right) \in {\mathcal  G} \left( \bar{\bar \nu} \left(\alpha_{n}, \beta_{n} \right) \right)$. 
We can show 
\(
\left( L_{\left( 2,0 \right)}^{-1}\bar{\bar x}, 1^* \right), 
\)
\(
\left( L_{\left( 2,0 \right)}^{-1}\bar{\bar x}, 2^* \right)
\in {\mathcal  G} \left( \bar{\bar \nu} \left(\alpha_{n}, \beta_{n} \right) \right)
\)
in a similar manner.

About (2):
For
\begin{equation}
\label{eq:difdif}
\left( \bar{\bar x}, i^{*} \right), 
\left( \bar{\bar y}, j^{*} \right) \in {\mathcal G} \left( \bar{\bar \nu} \left( \alpha_{n+1}, \beta_{n+1} \right) \right) 
\mbox{ satisfying }
\left( \bar{\bar x}, i^{*} \right) \neq 
\left( \bar{\bar y}, j^{*} \right),
\end{equation}
we will check whether
there exist a common unit square between 
$\Theta_{\left( 2, 0 \right)} 
\left( \bar{\bar x}, i^{*} \right)$ and 
$\Theta_{\left( 2, 0 \right)} 
\left( \bar{\bar y}, j^{*} \right)$.
There are six patterns of 
the combinations of $i^{*}$ and $j^{*}$ as \\
$\left\{
\left( 0^{*}, 0^{*} \right),
\left( 0^{*}, 1^{*} \right),
\left( 0^{*}, 2^{*} \right),
\left( 1^{*}, 1^{*} \right),
\left( 1^{*}, 2^{*} \right),
\left( 2^{*}, 2^{*} \right)
\right\}$.
\begin{enumerate}
\item[(i)]
Case of $\Theta_{\left( 2, 0 \right)} 
\left( \bar{\bar x}, 0^{*} \right)$ and 
$\Theta_{\left( 2, 0 \right)} 
\left( \bar{\bar y}, 0^{*} \right)$:
By 
(\ref{eq:dualSubst}),
\[
\begin{array}{rcl}
\Theta_{\left( 2, 0 \right)} \left( \bar{ \bar x}, 0^{*} \right)
& = & 
\left( L_{\left( 2,0 \right)}^{-1} \bar{\bar x}, 0^{*} \right) 
+
\left( L_{\left( 2,0 \right)}^{-1}  \left(
\bar{\bar x} + \bar{\bar e}_{2}\right) , 2^{*} \right) \\
\Theta_{\left( 2, 0 \right)} \left( \bar{ \bar y}, 0^{*} \right)
& = & 
\left( L_{\left( 2,0 \right)}^{-1} \bar{\bar y}, 0^{*} \right) 
+
\left( L_{\left( 2,0 \right)}^{-1}  \left(
\bar{\bar y} + \bar{\bar e}_{2}\right) , 2^{*} 
\right).
\end{array}
\]
Thus, if there exist a common square between 
$\Theta_{\left( 2, 0 \right)} 
\left( \bar{\bar x}, 0^{*} \right)$ and \\
$\Theta_{\left( 2, 0 \right)} 
\left( \bar{\bar y}, 0^{*} \right)$, 
then
\[\begin{array}{l}
\left( L_{\left( 2,0 \right)}^{-1} \bar{\bar x}, 0^{*} \right) 
=
\left( L_{\left( 2,0 \right)}^{-1} \bar{\bar y}, 0^{*} \right) 
\mbox{ or }\\
\left( L_{\left( 2,0 \right)}^{-1} \left( \bar{\bar x}
+ \bar{\bar e}_{2} \right), 
2^* \right) 
=
\left( L_{\left( 2,0 \right)}^{-1} \left( \bar{\bar y}
+ \bar{\bar e}_{2} \right),2^* \right) 
\end{array}
\]
holds which implies
$\bar{\bar x} = \bar{\bar y}$.  
This contradicts (\ref{eq:difdif}).
Therefore, there are no common unit 
squares
 between
 $\Theta_{\left( 2, 0 \right)} 
\left( \bar{\bar x}, 0^{*} \right)$ and 
$\Theta_{\left( 2, 0 \right)} 
\left( \bar{\bar y}, 0^{*} \right)$.
\item[(ii)]
Case of $\Theta_{\left( 2, 0 \right)} 
\left( \bar{\bar x}, 0^{*} \right)$ and 
$\Theta_{\left( 2, 0 \right)} 
\left( \bar{\bar y}, 1^{*} \right)$:
from
(\ref{eq:dualSubst}),
\[
\begin{array}{rcl}
\Theta_{\left( 2, 0 \right)} \left( \bar{ \bar x}, 0^{*} \right)
& = & 
\left( L_{\left( 2,0 \right)}^{-1} \bar{\bar x}, 0^{*} \right) 
+
\left( L_{\left( 2,0 \right)}^{-1}  \left(
\bar{\bar x} + \bar{\bar e}_{2}\right) , 2^{*} \right), \\
\Theta_{\left( 2, 0 \right)} \left( \bar{ \bar y}, 1^{*} \right)
& = & 
\left( L_{\left( 2,0 \right)}^{-1} \bar{\bar y}, 1^{*} \right).
\end{array}
\]
It is clear that there does not exist a common square between\\
$\Theta_{\left( 2, 0 \right)} 
\left( \bar{\bar x}, 0^{*} \right)$ and 
$\Theta_{\left( 2, 0 \right)} 
\left( \bar{\bar y}, 1^{*} \right)$. 
\item[(iii)]
Case of $\Theta_{\left( 2, 0 \right)} 
\left( \bar{\bar x}, 0^{*} \right)$ and 
$\Theta_{\left( 2, 0 \right)} 
\left( \bar{\bar y}, 2^{*} \right)$:
By
(\ref{eq:dualSubst}),
\[
\begin{array}{rcl}
\Theta_{\left( 2, 0 \right)} \left( \bar{ \bar x}, 0^{*} \right)
& = & 
\left( L_{\left( 2,0 \right)}^{-1} \bar{\bar x}, 0^{*} \right) 
+
\left( L_{\left( 2,0 \right)}^{-1}  \left(
\bar{\bar x} + \bar{\bar e}_{2}\right) , 2^{*} \right), \\
\Theta_{\left( 2, 0 \right)} \left( \bar{ \bar y}, 2^{*} \right)
& = & 
\left( L_{\left( 2,0 \right)}^{-1} \bar{\bar y}, 2^{*} \right).
\end{array}
\]
Thus, if there exist a common square between 
$\Theta_{\left( 2, 0 \right)} 
\left( \bar{\bar x}, 0^{*} \right)$ and \\
$\Theta_{\left( 2, 0 \right)} 
\left( \bar{\bar y}, 2^{*} \right)$, then
\[
\left( L_{\left( 2,0 \right)}^{-1} \left( \bar{\bar x}
+ \bar{\bar e}_{2} \right), 2^{*} \right) 
=
\left( L_{\left( 2,0 \right)}^{-1} \bar{\bar y}
, 2^{*} \right),
\]
i.e., 
\begin{equation}
\label{eq:91}
\bar{\bar x} = \bar{\bar y}- \bar{\bar e}_{2}.
\end{equation}
On the other hand, from 
$\left( \bar{\bar x}, 0^{*} \right), \left( \bar{\bar y}, 2^{*} \right) 
\in {\mathcal G} \left( \bar{\bar \nu} \left( \alpha_{n+1}, \beta_{n+1} \right) \right) $, 
we have
\begin{eqnarray}
\left< 
\bar{\bar x}, \bar{\bar v} \left( \alpha_{n+1}, \beta_{n+1} \right) 
\right> >0 
\label{eq:92}\\
\left< 
\bar{\bar y}-\bar{\bar e}_{2}, \bar{\bar v} \left( \alpha_{n+1}, \beta_{n+1} \right) 
\right> \leq 0.
\label{eq:93}
\end{eqnarray}
However,
using (\ref{eq:91}) and 
(\ref{eq:93}), we get 
\[
\left< \bar{\bar x}, \bar{\bar \nu} \left( \alpha_{n+1}, \beta_{n+1} \right)\right> 
=\left< \bar{\bar y}-\bar{\bar e}_{2},  \bar{\bar \nu} \left( \alpha_{n+1}, \beta_{n+1} \right)\right> 
\leq 0,
\]
which
contradicts (\ref{eq:92}).
Therefore, there are no common unit 
squares between
 $\Theta_{\left( 2, 0 \right)} 
\left( \bar{\bar x}, 0^{*} \right)$ and 
$\Theta_{\left( 2, 0 \right)} 
\left( \bar{\bar y}, 2^{*} \right)$.
\end{enumerate}
The other cases can be proved analogously.

About (3): 
For$ \left( \bar{ \bar z}, i^{*} \right) \in {\mathcal  G} \left( \bar{\bar \nu} \left(\alpha_{n}, \beta_{n} \right) \right)$,
we have
\begin{eqnarray}
\left< \bar{\bar z}, \bar{\bar v} \left( \alpha_{n}, \beta_{n} \right)  \right> >0
\label{eq:94}\\
\left< \bar{\bar z}- \bar{\bar e}_{i}, \bar{\bar v} \left( \alpha_{n}, \beta_{n} \right)  \right> \leq 0.
\label{eq:95}
\end{eqnarray}

\begin{enumerate}
\item[(i)]
For $\left( \bar{\bar z}, 0^{*} \right) \in {\mathcal  G} \left( \bar{\bar \nu} \left(\alpha_{n}, \beta_{n} \right) \right)$,
there exists
$\left( \bar{\bar x}, 0^{*} \right) \in 
{\mathcal G} \left( \bar{\bar \nu} \left( \alpha_{n+1}, \beta_{n+1} \right) \right) $ satisfying
\begin{equation}
\label{eq:96}
\bar{\bar x} = L_{\left( 2, 0 \right)} \bar{\bar z}
\end{equation}
such that
$\Theta_{\left( 2, 0 \right)} 
\left( \bar{\bar x}, 0^{*} \right)  \succ 
\left( \bar{\bar z}, 0^{*} \right)$. 
In fact, from 
(\ref{eq:dualSubst})
and 
 (\ref{eq:96}), it follows
\begin{eqnarray*}
\Theta_{\left( 2, 0 \right)} 
\left( \bar{\bar x}, 0^{*} \right) & = & 
\left( L_{\left( 2,0 \right)}^{-1} \bar{\bar x}, 0^{*} \right)+
\left( L_{\left( 2,0 \right)}^{-1} \left( \bar{\bar x}  + \bar{\bar e}_{2} \right), 2^{*} \right)\\
&=&
\left( \bar{\bar z}, 0^{*} \right)
+
\left( L_{\left( 2,0 \right)}^{-1} \left( \bar{\bar x}  + \bar{\bar e}_{2} \right), 2^{*} \right) \succ
\left( \bar{\bar z}, 0^{*} \right).
\end{eqnarray*}
By  (\ref{eq:96}), Lemma \ref{lem:u}, and (\ref{eq:94}), we get
\begin{eqnarray*}
\left< \bar{\bar x}, \bar{ \bar \nu} \left( \alpha_{n+1}, \beta_{n+1} \right) \right>
&=&
\left<L_{\left( 2, 0 \right)} \bar{\bar z}, \bar{ \bar \nu} \left( \alpha_{n+1}, \beta_{n+1} \right) \right>\\
&=&
\left<\bar{\bar z}, {}^{t} L_{\left( 2, 0 \right)} \bar{ \bar \nu} \left( \alpha_{n+1}, \beta_{n+1} \right) \right>\\
&=&c_{n} \left<\bar{\bar z}, \bar{\bar v} \left( \alpha_{n}, \beta_{n} \right)  \right>
>0.
\end{eqnarray*}
By  (\ref{eq:96}), Lemma \ref{lem:u}, and (\ref{eq:95}), we get
\begin{eqnarray*}
\lefteqn{
\left< \bar{\bar x}-\bar{\bar e}_{0}, \bar{ \bar \nu} \left( \alpha_{n+1}, \beta_{n+1} \right) \right>}\\
&=&
\left<L_{\left( 2, 0 \right)} \bar{\bar z} - \bar{\bar e}_{0}, \bar{ \bar \nu} \left( \alpha_{n+1}, \beta_{n+1} \right) \right>
=
\left<L_{\left( 2, 0 \right)} \left( \bar{\bar z} - \bar{\bar e}_{0} \right), \bar{ \bar \nu} \left( \alpha_{n+1}, \beta_{n+1} \right) \right>\\
&=&
\left<\bar{\bar z}-\bar{\bar e}_{0}, {}^{t} L_{\left( 2, 0 \right)} \bar{ \bar \nu} \left( \alpha_{n+1}, \beta_{n+1} \right) \right>
=
c_{n}
\left<\bar{\bar z}- \bar{\bar e}_{0},  \bar{\bar v} \left( \alpha_{n}, \beta_{n} \right)  \right>
\leq 0.
\end{eqnarray*}

\item[(ii)]
For $\left( \bar{\bar z}, 1^{*} \right) \in {\mathcal  G} \left( \bar{\bar \nu} \left(\alpha_{n}, \beta_{n} \right) \right)$,
there exists
$\left( \bar{\bar x}, 1^{*} \right) \in 
{\mathcal G} \left( \bar{\bar \nu} \left( \alpha_{n+1}, \beta_{n+1} \right) \right) $ satisfying
\begin{equation}
\label{eq:97}
\bar{\bar x} = L_{\left( 2, 0 \right)} \bar{\bar z}
\end{equation}
such that
$\Theta_{\left( 2, 0 \right)} 
\left( \bar{\bar x}, 1^{*} \right)  \succ 
\left( \bar{\bar z}, 1^{*} \right)$. 
In fact, from 
(\ref{eq:dualSubst}) 
and (\ref{eq:97}), it follows
\[
\Theta_{\left( 2, 0 \right)} 
\left( \bar{\bar x}, 1^{*} \right) = 
\left( L_{\left( 2,0 \right)}^{-1} \bar{\bar x}, 1^{*} \right)
=
\left( \bar{\bar z}, 1^{*} \right)
\succ 
\left( \bar{\bar z}, 1^{*} \right).
\]
By (\ref{eq:97}), Lemma \ref{lem:u}, and 
(\ref{eq:94}), we get
\begin{eqnarray*}
\lefteqn{
\left< \bar{\bar x}, \bar{ \bar \nu} \left( \alpha_{n+1}, \beta_{n+1} \right) \right>}\\
&=&
\left<L_{\left( 2, 0 \right)} \bar{\bar z}, \bar{ \bar \nu} \left( \alpha_{n+1}, \beta_{n+1} \right) \right>
=
\left<\bar{\bar z}, {}^{t} L_{\left( 2, 0 \right)} \bar{ \bar \nu} \left( \alpha_{n+1}, \beta_{n+1} \right) \right>\\
&=&
c_{n}  \left<\bar{\bar z},\bar{\bar v} \left( \alpha_{n}, \beta_{n} \right)  \right>
>0.
\end{eqnarray*}
By (\ref{eq:97}), Lemma \ref{lem:u}, and 
(\ref{eq:95}), we get
\begin{eqnarray*}
\lefteqn{
\left< \bar{\bar x}-\bar{\bar e}_{1}, \bar{ \bar \nu} \left( \alpha_{n+1}, \beta_{n+1} \right) \right>}\\
&=&
\left<L_{\left( 2, 0 \right)} \bar{\bar z} - \bar{\bar e}_{1}, \bar{ \bar \nu} \left( \alpha_{n+1}, \beta_{n+1} \right) \right>
=
\left<L_{\left( 2, 0 \right)} \left( \bar{\bar z} - \bar{\bar e}_{1} \right), \bar{ \bar \nu} \left( \alpha_{n+1}, \beta_{n+1} \right) \right>\\
&=&
\left<\bar{\bar z}-\bar{\bar e}_{1}, {}^{t} L_{\left( 2, 0 \right)} \bar{ \bar \nu} \left( \alpha_{n+1}, \beta_{n+1} \right) \right>
=
c_{n}\left<\bar{\bar z}- \bar{\bar e}_{1}, \bar{\bar v} \left( \alpha_{n}, \beta_{n} \right)  \right>
\leq 0.
\end{eqnarray*}

\item[(iii)]
For $\left( \bar{\bar z}, 2^{*} \right) \in
{\mathcal  G} \left( \bar{\bar \nu} \left(\alpha_{n}, \beta_{n} \right) \right)$,
\begin{enumerate}
\item[(a)] 
if 
\begin{equation}
\label{eq:98}
\left< 
\bar{\bar z}-\bar{\bar e}_{2}+\bar{\bar e}_{0}, 
\bar{\bar \nu} \left( \alpha_{n}, \beta_{n} \right)\right> >0,
\end{equation} then 
there exists
$\left( \bar{\bar x}, 0^{*} \right) \in 
{\mathcal G} \left( \bar{\bar \nu} \left( \alpha_{n+1}, \beta_{n+1} \right) \right) $ satisfying
\begin{equation}
\label{eq:99}
\bar{\bar x} = L_{\left( 2, 0 \right)} \bar{\bar z}-
\bar{\bar e}_{2}
\end{equation}
such that
$\Theta_{\left( 2, 0 \right)} 
\left( \bar{\bar x}, 0^{*} \right)  \succ 
\left( \bar{\bar z}, 2^{*} \right)$;
\item[(b)] 
if \begin{equation}
\label{eq:100}
\left< 
\bar{\bar z}-\bar{\bar e}_{2}+\bar{\bar e}_{0}, 
\bar{\bar \nu} \left( \alpha_{n}, \beta_{n} \right)\right> \leq 0,
\end{equation} then 
there exists
$\left( \bar{\bar x}, 2^{*} \right) \in 
{\mathcal G} \left( \bar{\bar \nu} \left( \alpha_{n+1}, \beta_{n+1} \right) \right) $ satisfying
\begin{equation}
\label{eq:101}
\bar{\bar x} = L_{\left( 2, 0 \right)} \bar{\bar z}
\end{equation}
such that
$\Theta_{\left( 2, 0 \right)} 
\left( \bar{\bar x}, 2^{*} \right)  \succ 
\left( \bar{\bar z}, 2^{*} \right)$.
\end{enumerate}

About (a): From 
(\ref{eq:dualSubst})
together with (\ref{eq:99}), it follows
\begin{eqnarray*}
\Theta_{\left( 2, 0 \right)} 
\left( \bar{\bar x}, 0^{*} \right) &=& 
\left( L_{\left( 2,0 \right)}^{-1} \bar{\bar x}, 0^{*} \right)
+
\left( L_{\left( 2,0 \right)}^{-1} 
\left( \bar{\bar x} + \bar{\bar e}_{2} \right), 2^{*} \right)\\
&=&
\left( L_{\left( 2,0 \right)}^{-1} \bar{\bar x}, 0^{*} \right)+
\left( \bar{\bar z}, 2^{*} \right)
\succ 
\left( \bar{\bar z}, 2^{*} \right).
\end{eqnarray*}
By (\ref{eq:99}), Lemma 
\ref{lem:u}, and (\ref{eq:98}), we get,
\begin{eqnarray*}
\left< \bar{\bar x}, \bar{ \bar \nu} \left( \alpha_{n+1}, \beta_{n+1} \right) \right>
&=&
\left<L_{\left( 2, 0 \right)} \bar{\bar z}-\bar{\bar e}_{2}, \bar{ \bar \nu} \left( \alpha_{n+1}, \beta_{n+1} \right) \right>\\
&=&
\left<L_{\left( 2, 0 \right)} 
\left( \bar{\bar z}-\bar{\bar e}_{2}+\bar{\bar e}_{0}\right), \bar{ \bar \nu} \left( \alpha_{n+1}, \beta_{n+1} \right) \right>\\
&=&
\left<\bar{\bar z}-\bar{\bar e}_{2}+\bar{\bar e}_{0}, {}^{t} L_{\left( 2, 0 \right)} \bar{ \bar \nu} \left( \alpha_{n+1}, \beta_{n+1} \right) \right>\\
&=&
c_{n} 
\left<\bar{\bar z}-\bar{\bar e}_{2}+\bar{\bar e}_{0}, \bar{\bar v} \left( \alpha_{n}, \beta_{n} \right)  \right>
>0.
\end{eqnarray*}
By (\ref{eq:97}), Lemma \ref{lem:u}, and 
(\ref{eq:95}),
\begin{eqnarray*}
\left< \bar{\bar x}-\bar{\bar e}_{0}, \bar{ \bar \nu} \left( \alpha_{n+1}, \beta_{n+1} \right) \right>
&=&
\left<L_{\left( 2, 0 \right)} \bar{\bar z} - \bar{\bar e}_{2}- \bar{\bar e}_{0}, \bar{ \bar \nu} \left( \alpha_{n+1}, \beta_{n+1} \right) \right>\\
&=&
\left<L_{\left( 2, 0 \right)} \left( \bar{\bar z} - \bar{\bar e}_{2} \right), \bar{ \bar \nu} \left( \alpha_{n+1}, \beta_{n+1} \right) \right>\\
&=&
\left<\bar{\bar z}-\bar{\bar e}_{2}, {}^{t} L_{\left( 2, 0 \right)} \bar{ \bar \nu} \left( \alpha_{n+1}, \beta_{n+1} \right) \right>\\
&=&
c_{n} \left<\bar{\bar z}- \bar{\bar e}_{2}, \bar{\bar v} \left( \alpha_{n}, \beta_{n} \right)  \right>
\leq 0.
\end{eqnarray*}

About (b): From 
(\ref{eq:dualSubst})
together with (\ref{eq:101}), it follows 
\[
\Theta_{\left( 2, 0 \right)} 
\left( \bar{\bar x}, 2^{*} \right) = 
\left( L_{\left( 2,0 \right)}^{-1} \bar{\bar x}, 2^{*} \right)
=
\left( \bar{\bar z}, 2^{*} \right)
\succ 
\left( \bar{\bar z}, 2^{*} \right).
\]
By (\ref{eq:101}), Lemma 
\ref{lem:u}, and (\ref{eq:94}), we get,
\begin{eqnarray*}
\lefteqn{\left< \bar{\bar x}, \bar{ \bar \nu} \left( \alpha_{n+1}, \beta_{n+1} \right) \right>}\\
&=&
\left<L_{\left( 2, 0 \right)} \bar{\bar z}, \bar{ \bar \nu} \left( \alpha_{n+1}, \beta_{n+1} \right) \right>=
\left<\bar{\bar z}, {}^{t} L_{\left( 2, 0 \right)} \bar{ \bar \nu} \left( \alpha_{n+1}, \beta_{n+1} \right) \right>\\
&=&
c_{n} \left<\bar{\bar z}, \bar{\bar v} \left( \alpha_{n}, \beta_{n} \right)  \right>
>0.
\end{eqnarray*}
By (\ref{eq:101}), Lemma 
\ref{lem:u}, and (\ref{eq:100}), we get,
\begin{eqnarray*}
\left< \bar{\bar x}-\bar{\bar e}_{2}, \bar{ \bar \nu} \left( \alpha_{n+1}, \beta_{n+1} \right) \right>
&=&
\left<L_{\left( 2, 0 \right)} \bar{\bar z} - \bar{\bar e}_{2}, \bar{ \bar \nu} \left( \alpha_{n+1}, \beta_{n+1} \right) \right>\\
&=&
\left<L_{\left( 2, 0 \right)} \left( \bar{\bar z} - \bar{\bar e}_{2}
+ \bar{\bar e}_{0}
 \right), \bar{ \bar \nu} \left( \alpha_{n+1}, \beta_{n+1} \right) \right>\\
&=&
\left<\bar{\bar z}-\bar{\bar e}_{2}+ \bar{\bar e}_{0}, {}^{t} L_{\left( 2, 0 \right)} \bar{ \bar \nu} \left( \alpha_{n+1}, \beta_{n+1} \right) \right>\\
&=&
\left<\bar{\bar z}- \bar{\bar e}_{2}+ \bar{\bar e}_{0}, ,c_{n} \bar{\bar v} \left( \alpha_{n}, \beta_{n} \right)  \right>
\leq
0.
\end{eqnarray*}
\end{enumerate}
By the arguments (i)--(iii), we obtain the 
assertion (3).

\noindent
Other cases can be  proved analogously.
\end{proof}

\begin{corollary}
\label{cor:stur}
Let $K$ be a real cubic field and  $(\alpha_{0}, \beta_{0})\in  \Delta_{K}$. 
Moreover, we assume that the sequence 
$\left\{ \left( \alpha_{n}, \beta_{n}, \varepsilon_n \right)\right\}_{n=0,1,2,\ldots}$ is periodic 
with a period of length $p$,
 i.e.,
 \begin{equation}
 \label{eq:periodic}
\exists m  \geq 0,  \exists p \geq 1: 
\left( \alpha_{m}, \beta_{m}, \varepsilon_m \right) = 
\left( \alpha_{m+p}, \beta_{m+p}, \varepsilon_{m+p} \right).
\end{equation}
Then, the stepped surface ${\mathscr S} \left(
\bar{\bar \nu} \left( \alpha_0, \beta_0 \right) \right)$ can be presented by 
\begin{equation}
\label{eq:gousei}
{}_s \Psi_g \left( {\mathscr S} \left(\bar{\bar \nu} \left( \alpha_0, \beta_0 \right) \right) \right) = 
\Theta_{\varepsilon_0} 
\circ 
\Theta_{\varepsilon_1} \circ
\cdots
\circ 
\Theta_{\varepsilon_{m-1}} \left( 
{}_s \Psi_g  \left( 
{\mathscr S}
\left( {\bar{\bar \nu}} \left( \alpha_m, \beta_{m}  \right) \right)
\right)
\right)
\end{equation}
and the stepped surface 
${\mathscr S}
\left( {\bar{\bar \nu}} \left( \alpha_m, \beta_{m}  \right) \right)$
can be characterized as the fixed point of the tiling 
substitution 
\(
\Theta= \Theta_{\varepsilon_m} \Theta_{\varepsilon_{m+1}} 
\cdots
\Theta_{\varepsilon_{m+p-1}}
\) 
by
\begin{equation}
\label{eq:Theta}
\Theta \left( 
{}_s \Psi_g \left(
{\mathscr S}
\left( {\bar{\bar \nu}} \left( \alpha_m, \beta_{m}  \right) \right)
 \right) \right)= 
{}_s \Psi_g 
\left( 
{\mathscr S} \left( {\bar{\bar \nu}} \left( \alpha_m, \beta_{m}  \right)
\right)
 \right) .
\end{equation}
\end{corollary}

\begin{remark}
Notice that by the bijectivity of  $\Theta_{\varepsilon_n}$,
the  right-hand side (resp., left-hand side)
of (\ref{eq:gousei}) (resp., (\ref{eq:Theta}))
can be defined by extending the finite sum of squares to an
infinite sum. We can do the same for
${}_s \Psi_g$ and ${}_g \Psi_s$.
\end{remark}

\begin{proof}
It is clear that (\ref{eq:Theta}) is valid by 
(\ref{eq:periodic}).
Moreover, by the bijectivity of $\Theta_{\varepsilon_n}$, $n=0,1, \ldots, m-1$, we get (\ref{eq:gousei}).
\end{proof}

Let ${\mathcal U}$ and 
${\mathcal U'}$ be fundamental patches
\[
{\mathcal U}:=\sum_{i=0}^{2}
\left( \bar{\bar e}_{i}, i^{*} \right), \
{\mathcal U'}:=  \sum_{i=0}^{2}
\left( \bar{\bar 0}, i^{*} \right).
\]
We put

\begin{eqnarray*}
\gamma_{n} &:=& 
\Theta_{\varepsilon_0  }
\ldots
\Theta_{\varepsilon_{n-2}}
\Theta_{\varepsilon_{n-1}}
\left( {\mathcal U} \right)
\quad \mbox{ for } 
{\mathcal U} \in {\mathcal G} \left( \bar{ \bar{v}} \left( \alpha_{n}, \beta_n \right) \right),\\
\mbox{(}\mbox{resp., } \gamma_{n}'  &:=& 
\Theta_{\varepsilon_0 }
\ldots
\Theta_{\varepsilon_{n-2}}
\Theta_{\varepsilon_{n-1}}
\left( {\mathcal U}' \right)
\quad \mbox{ for } {\mathcal U'} \in 
{\mathcal G}' \left( \bar{ \bar{v}} \left( \alpha_{n}, \beta_n \right)
\right)
\mbox{)},
\end{eqnarray*}
which is a sequence of 
patches of 
$\mathscr{S}
\left( \bar{\bar \nu} \left( \alpha_{0}, \beta_0 \right) \right)
$
(resp., $\mathscr{S}' \left( \bar{\bar \nu} \left( \alpha_{0}, \beta_0 \right)
\right)$).
Then we have the following.

\begin{corollary}

\begin{enumerate}
\item[(1)]
The difference of 
$\gamma_{n}$ and $\gamma_{n}'$ is that
\[
{}_g \Psi_s \left( \gamma_n \right) \backslash
{}_g \Psi_s  \left( \gamma_n' \right) 
={}_g \Psi_s  \left( {\mathcal U} \right), \quad
{}_g \Psi_s \left( \gamma_n' \right) \backslash
{}_g \Psi_s  \left( \gamma_n \right) 
=
{}_g \Psi_s  \left( {\mathcal U}' \right) 
\]
\item[(2)]
$\gamma_{n} \prec \gamma_{n+1}$ for all $n$.\\
\item[(3)] 
$\bigcup_{n=0}^{\infty} {}_g \Psi_s\left(
\gamma_{n}  \right)
\subset
{\mathscr S}\left( \bar{\bar \nu} \left( \alpha_0, \beta_0 \right) \right)$. 
\end{enumerate}
\end{corollary}

\section{Examples}
We give some examples.
\begin{example}[$K$ is not totally real]
\label{ex:1}
Let $\delta$ be the real root of
$x^3-2$, $K={\Bbb Q}(\delta)$ and 
$\alpha=2/3-2\delta/3+\delta^2/6,$
$\beta=2/3+\delta/3-\delta^2/3$.
Then $\left(
\alpha,\beta \right)
\in \Delta_K$ and 
the expansion
$\left\{ \varepsilon_n \right\}^{\infty}_{n=0}$
of 
$\left( \alpha, \beta \right)$
obtained by our 
continued fraction algorithm is 
given by 
\begin{eqnarray*}
\left\{ \varepsilon_n \right\}^{\infty}_{n=0}
&=&\stackrel{*}{(2,0)},(0,2),(2,1),(2,1),(0,1),(1,0),(0,2),(0,2),(1,2),(2,1),\\
& & (1,0),\stackrel{*}{(1,0)},\ldots ,
\end{eqnarray*}
and 
the eigenvalue
$\lambda(\alpha,\beta)>1$ 
coming from the period of the expansions is a Pisot number
 with
 $x^3-57x^2+3x-1$ as its minimal polynomial (see Figure \ref{fig:4}).
On the other hand, we can observe the 
explosion phenomenon 
(as in the example in \cite{TY})
related to the expansions obtained by
the Jacobi-Perron and the modified Jacobi-Perron
algorithms;
consequently, we can not expect the 
periodicity of the expansions.      
\end{example}

\begin{figure*}[hbtp]
\begin{center}
\begin{minipage}{3cm}
\begin{center}
\includegraphics[width=2cm]
{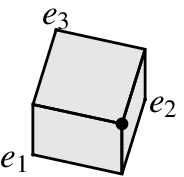}\\
$n=0$
\end{center}
\end{minipage}
\quad \quad \quad
\begin{minipage}{5cm}
\begin{center}
\includegraphics[width=4.5cm]
{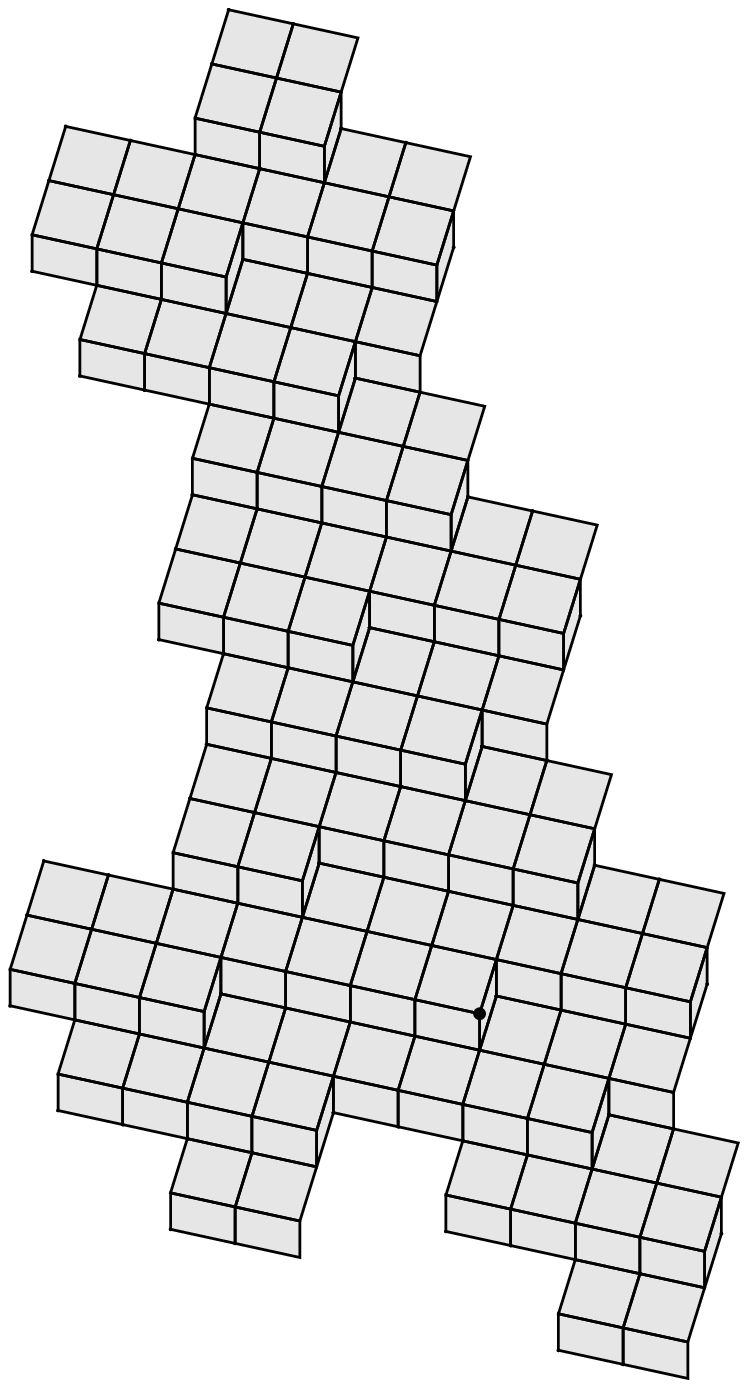}\\
$n=1$
\end{center}
\end{minipage}
\begin{center}
\includegraphics[width=12cm]{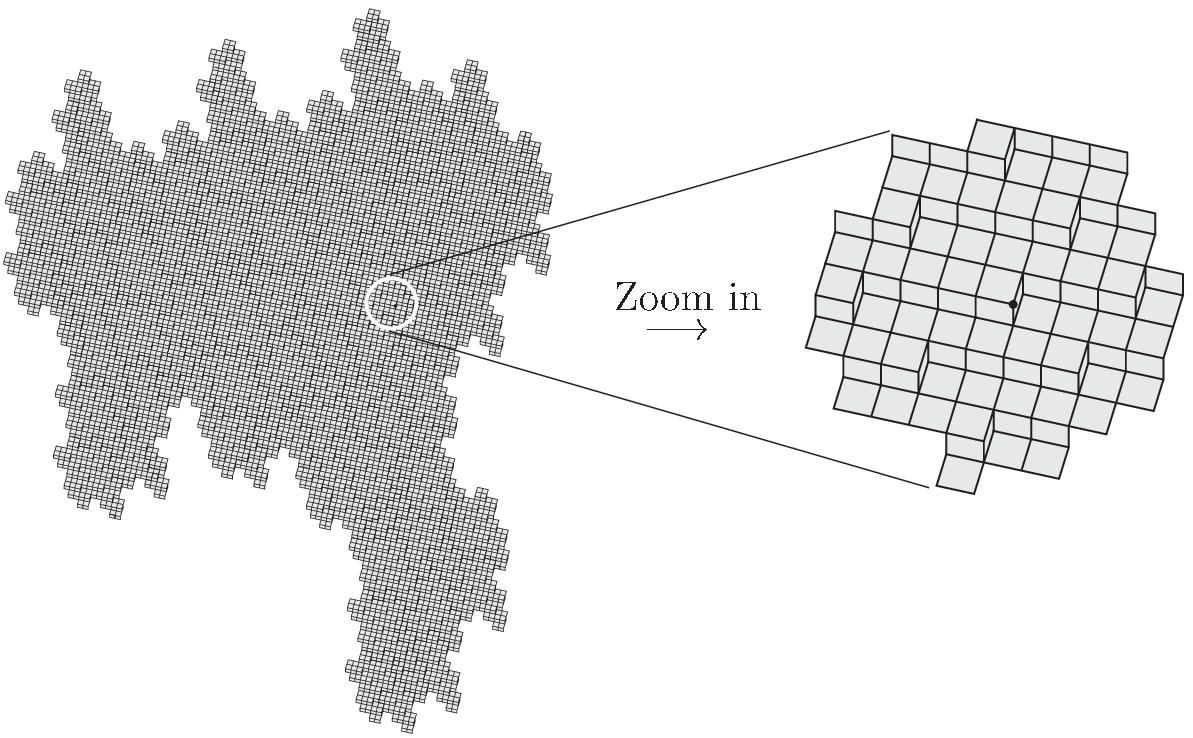}\\
$n=2$
\end{center}

\end{center}
\caption{$\displaystyle
\left(
\Theta_{{(2,0)}}
\Theta_{{(0,2)}}
\Theta_{{(2,1)}}
\Theta_{{(2,1)}}
\Theta_{{(0,1)}}
\Theta_{{(1,0)}}
\Theta_{{(0,2)}}
\Theta_{{(0,2)}}
\Theta_{{(1,2)}}
\Theta_{{(2,1)}}
\Theta_{{(1,0)}}
\Theta_{{(1,0)}}
\right)^{n}
\left( {\mathcal U} \right)$ in Example \ref{ex:1}
where the point is located at $\left( 1,1,1 \right)$.}
\label{fig:4}
\end{figure*}

\begin{example}[$K$ is totally real]
\label{ex:3}
\begin{enumerate} 
\item[(1)]
Let $\delta$ be the root of  $x^3-6x^2+7x-1$ with $\delta>4$,
$K=\mathbb{Q} \left( \delta \right)$  
and
$\alpha=-1/3-4\delta/3+\delta^2/3$,
 $\beta=-2+5\delta-\delta^2$.
Then, $\left(
\alpha,\beta \right)
\in \Delta_K$ and 
$^{t} \left( 1 - \alpha-\beta, \alpha, \beta  \right)$
is an eigenvector of 
$M_{(1,0)}M_{(0,1)}M_{(2,0)}M_{(0,2)}M_{(0,1)}$
with respect to its eigenvalue $\delta$.
We note that $\delta$ is not a Pisot number
(see Figure \ref{fig:piropiro2}).

\item[(2)]
On the other hand, we have the expansion of 
$\left(
\alpha,\beta
\right)$ 
as follows:
\begin{align*}
\left\{ \varepsilon_n \right\}_{n=0}^{\infty}=&
(1, 0), (0, 2), (2, 1), (2, 1), (2, 1), (1, 0), (1, 0), \\
&\stackrel{*}{(1, 0)},(0, 2), (2, 0), (1, 0), 
(2, 0), (0, 1), (1, 0),
\stackrel{*}{(2, 0)}, \ldots ,
\end{align*}
and $\lambda(\alpha,\beta)$ is a Pisot number having
its minimal  
polynomial $x^3-13x^2+10x-1$
(see Figure \ref{fig:even2Period1kai}).
\end{enumerate}
\end{example}

\begin{figure*}[hbtp]
\begin{center}
\begin{minipage}{2cm}
\begin{center}
\includegraphics[width=1.8cm]{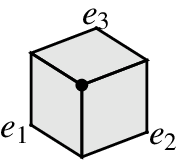}\\
$n=0$
\end{center}
\end{minipage}
\quad \quad \quad
\begin{minipage}{6cm}
\begin{center}
\includegraphics[width=6.5cm]{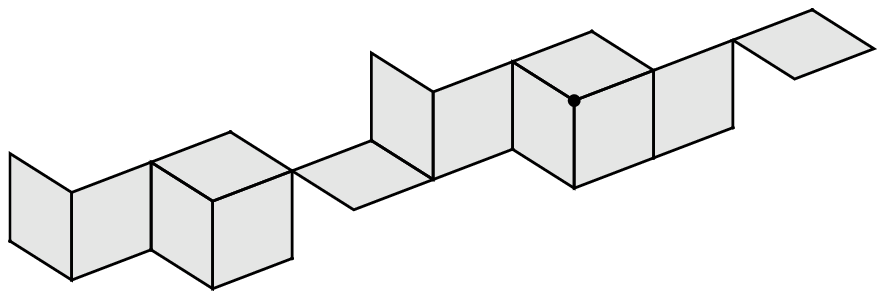}\\
$n=1$
\end{center}
\end{minipage}\\
\myvcenter{\includegraphics[width=14cm]{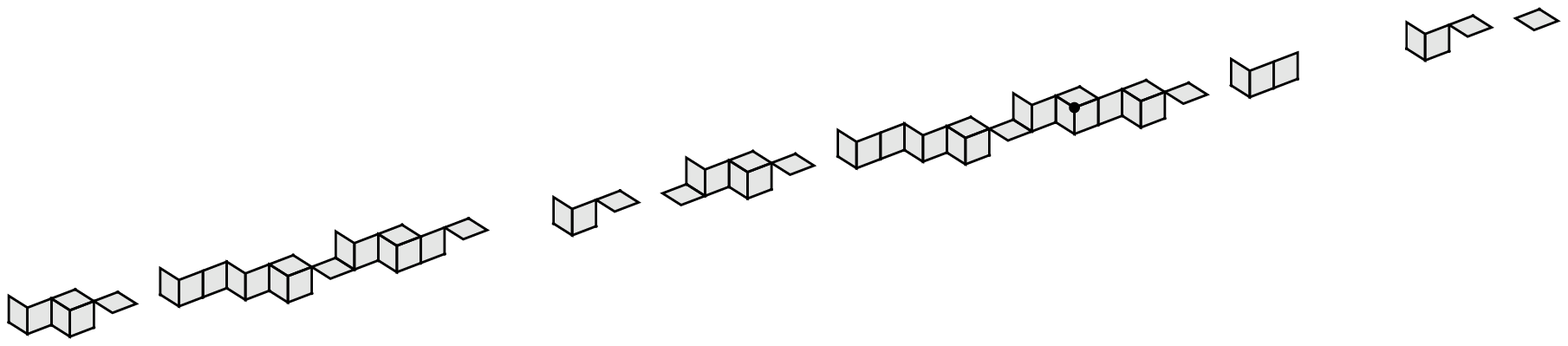}}\\
$n=2$
%

\end{center}
\caption{$\displaystyle
(\Theta_{(1,0)}\Theta_{(0,1)}\Theta_{(2,0)}
\Theta_{(0,2)} \Theta_{(0,1)})^n \left({\mathcal U} \right)$ in Example \ref{ex:3} (1).}
\label{fig:piropiro2}
\end{figure*}

\begin{figure*}[hbtp]
\begin{center}
\begin{minipage}{2cm}
\begin{center}
\includegraphics[width=1.8cm]{exam62no1-0step.eps}\\
$n=0$
\end{center}
\end{minipage}
\quad \quad \quad
\begin{minipage}{6cm}
\begin{center}
\includegraphics[width=7cm]{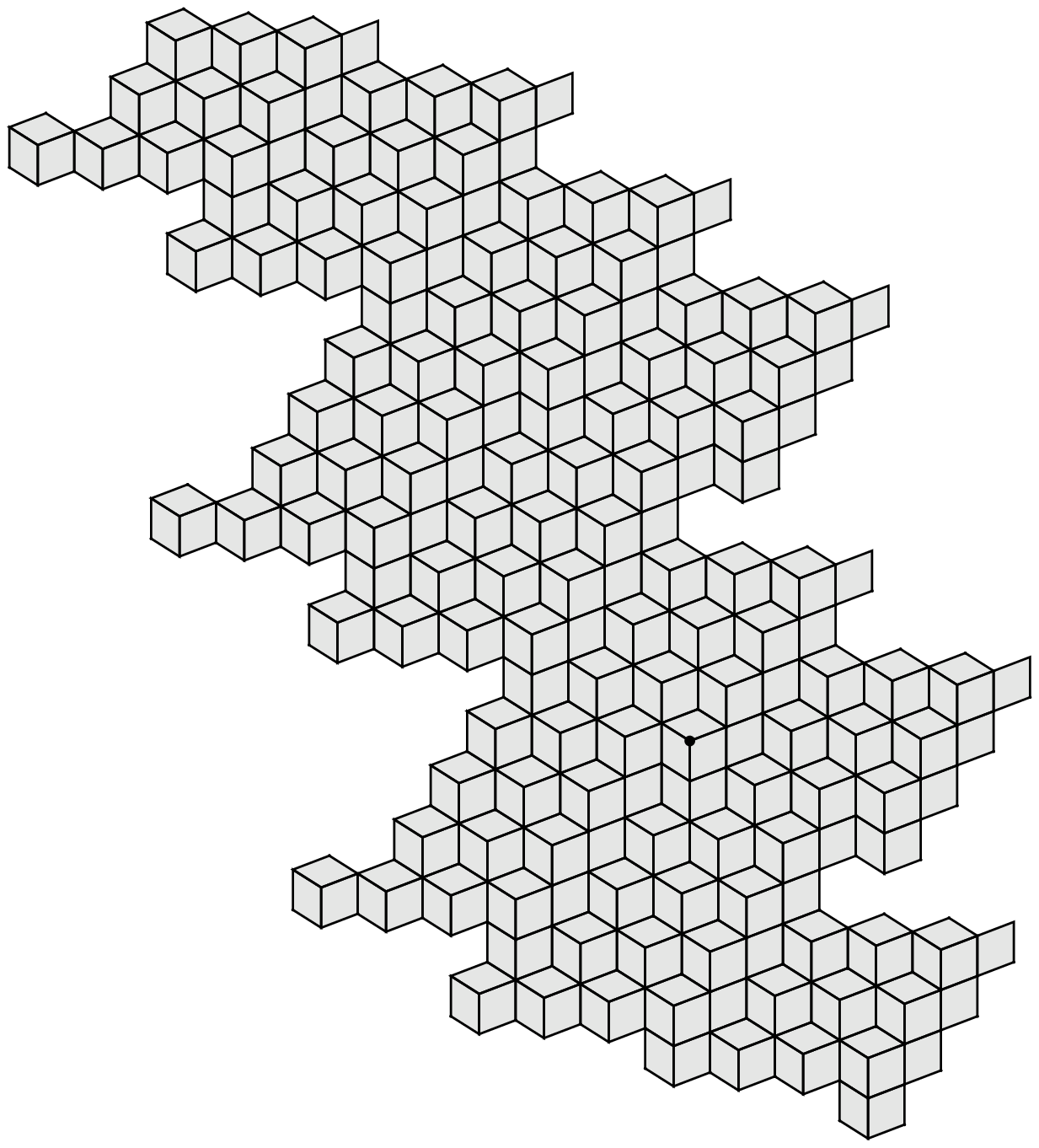}\\
$n=1$
\end{center}
\end{minipage}\\

\begin{center}
\includegraphics[width=12cm]{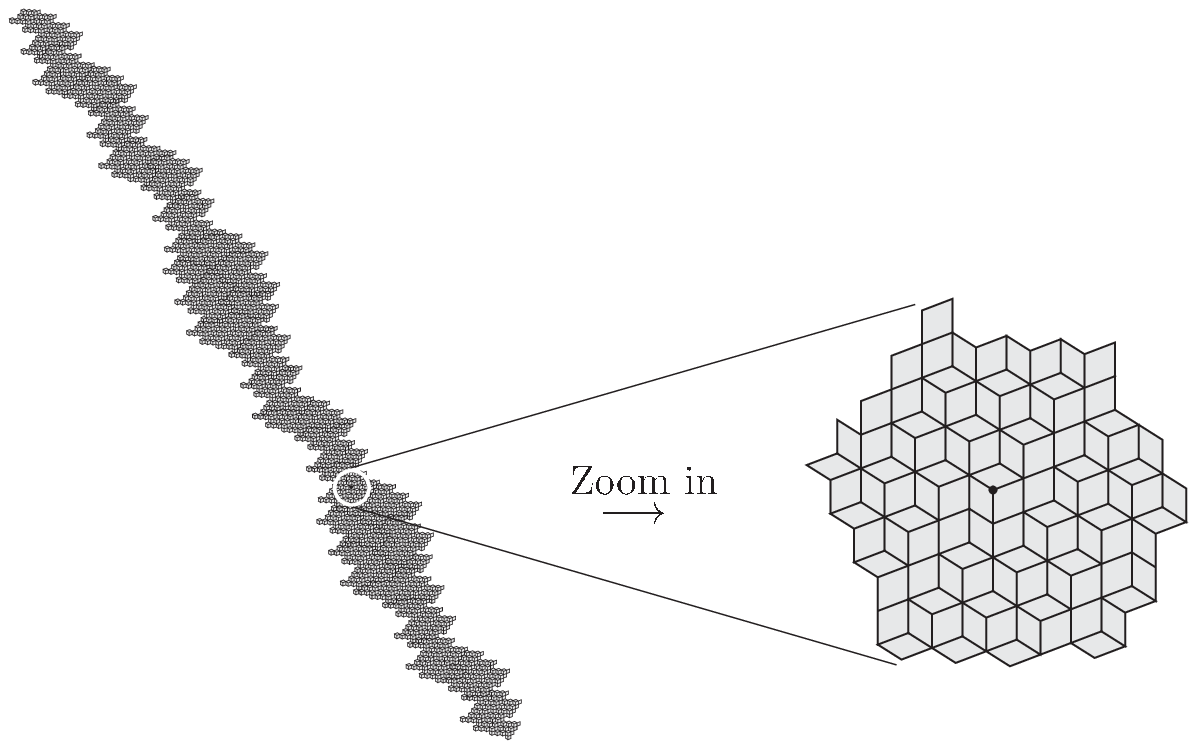}\\
$n=2$
\end{center}

%
\end{center}
\caption{$\displaystyle
\Theta_{(1,0)}
\Theta_{(0,2)}
\Theta_{(2,1)}
\Theta_{(2,1)}
\Theta_{(2,1)}
\Theta_{(1,0)}
\Theta_{(1,0)}
\left( 
\Theta_{(1,0)}
\Theta_{(0,2)}
\Theta_{(2,0)}
\Theta_{(1,0)}
\Theta_{(2,0)}
\Theta_{(0,1)}
\Theta_{(1,0)}
\Theta_{(2,0)}
\right)^{n}
\left( {\mathcal U} \right)$
in Example \ref{ex:3} (2).}
\label{fig:even2Period1kai}
\end{figure*}

\begin{example}[$K$ is not totally real]
\label{ex:4}
\begin{enumerate} 
\item[(1)] (Completely non-admissible)\\
Let $\delta$ be the real root of  $x^3-5x^2-2x-1$,
$K=\Bbb{Q}(\delta)$
and
$\alpha=11/5+9\delta/5-2\delta^2/5$,
 $\beta=-7/5+7\delta/5-\delta^2/5$.
Then, $(\alpha,\beta)\in \Delta_K$ and ${}^t(1-\alpha-\beta,\alpha,\beta)$
is an eigenvector of 
$M_{(0,1)}M_{(2,0)}M_{(1,2)}M_{(0,1)}M_{(2,0)}M_{(1,2)}$
with respect to its eivenvalue $\delta$.
We note that $(0,1)(2,0)$,
$(2,0)(1,2)$,
$(1,2)(0,1)$,
$(2,0)(1,2)$,
$(1,2)(0,1)$ are 
forbidden words given
in Table \ref{table:forbidden}
and $\delta$ is a Pisot number
(see Figure \ref{fig:tori2}).
\item[(2)]
We have the expansion of 
$\left(
\alpha,\beta
\right)$ as follows:
\begin{align*}
\left\{\varepsilon_n
\right\}_{n=0}^{\infty}=
\stackrel{*}{(0, 2)}, (2, 1), (1, 0), (1, 0), (0, 2), (2, 1),
(2, 1), (1, 0),\stackrel{*}{(0, 2)}, \ldots ,
\end{align*}
and $\lambda
\left(
\alpha,\beta
\right)$ is a Pisot number having  $x^3-29x^2-6x-1$
as its minimal polynomial
(see Figure  \ref{fig:tori6kai}).    
\end{enumerate}
\end{example}

\begin{figure*}[hbtp]
\begin{center}
\begin{minipage}{5cm}
\begin{center}
\includegraphics[width=1.5cm]{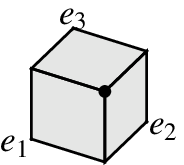}\\
$n=0$\\
\vspace{1cm}
\end{center}
\end{minipage}
\begin{minipage}{4cm}
\begin{center}
\includegraphics[width=4cm]{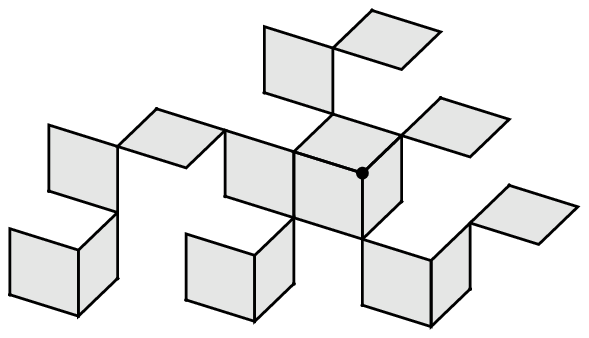}\\
$n=1$
\end{center}
\end{minipage}
\quad
\begin{minipage}{9cm}
\begin{center}
\includegraphics[width=9cm]{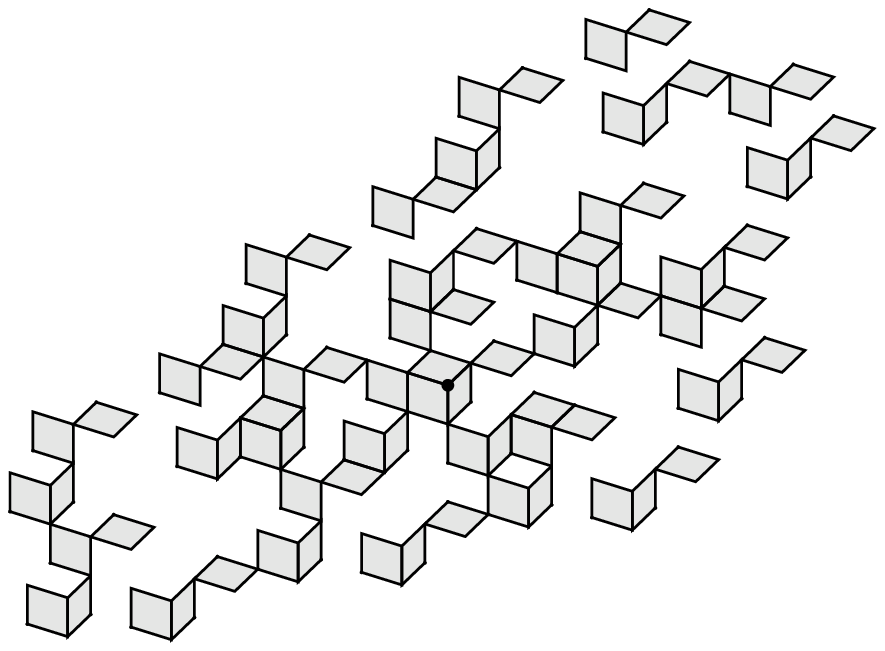}\\
$n=2$
\end{center}
\end{minipage}

\end{center}
\caption{$(
\Theta_{(0,1)}
\Theta_{(2,0)}
\Theta_{(1,2)}
\Theta_{(0,1)}
\Theta_{(2,0)}\Theta_{(1,2)} 
)^n \left({\mathcal U} \right)$ 
in Example \ref{ex:4} (1).}
\label{fig:tori2}
\end{figure*}

\begin{figure*}[hbtp]
\begin{center}
\begin{minipage}{2cm}
\begin{center}
\includegraphics[width=1.8cm]{exam63no1-0step.eps}\\
$n=0$
\end{center}
\end{minipage}
\quad \quad \quad
\begin{minipage}{6cm}
\begin{center}
\includegraphics[width=6cm]{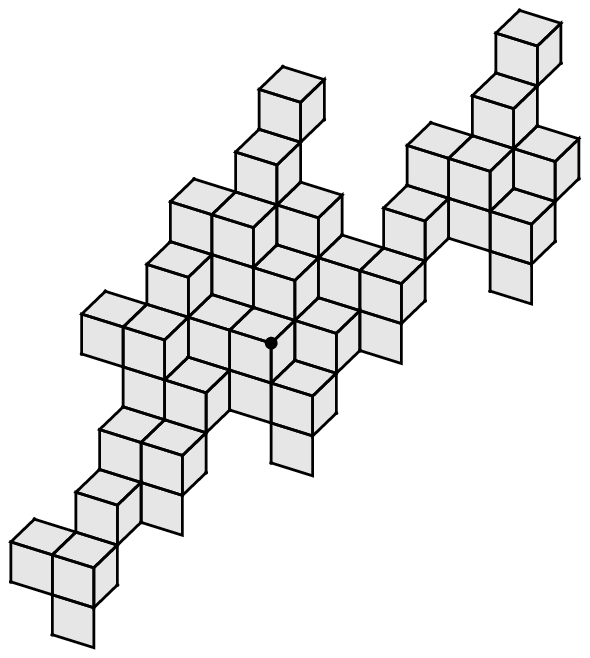}\\
$n=1$
\end{center}
\end{minipage}\\
\bigskip
\begin{center}
\includegraphics[width=12cm]{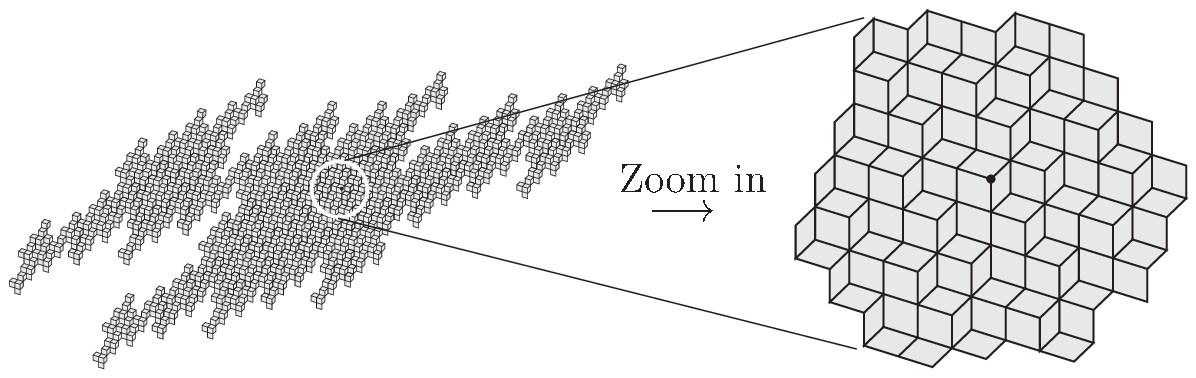}\\
$n=2$
\end{center}

%
\end{center}
\caption{$\left(
\Theta_{(0,2)}
\Theta_{(2,1)}
\Theta_{(1,0)}
\Theta_{(1,0)}
\Theta_{(0,2)}
\Theta_{(2,1)}
\Theta_{(2,1)}
\Theta_{(1,0)}
\Theta_{(0,2)}
\right)^{n}
\left( {\mathcal U} \right)$
in Example \ref{ex:4} (2).}
\label{fig:tori6kai}
\end{figure*}
\clearpage

\section{Conjectures}
\label{sec:con}

We give following conjectures which are supported by the numerical experiments. 

\begin{conjecture}
\label{con1}
Let $K$ be a real cubic field and $r=5/2$.
$\Delta_K= \Delta_{K, r}^{\mathcal{P}er}$ holds for any real cubic field $K$.
\end{conjecture}

\begin{conjecture}
\label{con2}
Let $K$ be a real cubic field 
 and let $\alpha^{(0)}, \alpha^{(1)}, \alpha^{(2)}$
be its positive ${\Bbb Q}$-basis with
\begin{align*}
\alpha=\frac{\alpha^{(1)}}{\alpha^{(0)}+\alpha^{(1)}+\alpha^{(2)}},\
\beta=\frac{\alpha^{(2)}}{\alpha^{(0)}+\alpha^{(1)}+\alpha^{(2)}}. 
\end{align*}
Let  $\left\{
\varepsilon_n \right\}_{n=0}^{\infty}$
be the expansion of the $(\alpha,\beta)$.
Suppose that $\varepsilon_{k+1},\ldots, \varepsilon_{k+l}$ 
is the period of the expansion.
Then, 
$M_{\varepsilon_{k+1}}\ldots M_{\varepsilon_{k+l}}$
has a Pisot number as its eigenvalue.
\end{conjecture}

The following theorem (Theorem 5 in 
Fernique \cite{F})
together with Conjectures \ref{con1}, \ref{con2} implies
Conjecture \ref{conclusion}.

\begin{theorem}[Fernique] 
\label{th:6.3}
Let ${\mathscr S} \left( \bar{\bar \alpha} \right)$ be a stepped surface of such that there exist two generalized substitutions $\Theta_{Prep}$ and 
$\Theta_{{\mathcal P}er}$ verifying:
\[
{\mathscr S} \left( \bar{\bar \alpha} \right) = \Theta_{Prep}
\left( {\mathscr S} \right) \mbox{ with } 
{\mathscr S} = \Theta_{{\mathcal P}er} 
\left( {\mathscr S} \right).
\]
If $\Theta_{{\mathcal P}er}$ is of Pisot type and 
bijective on $\mathscr S$, then there exists a finite patch 
$P$ of $\mathscr S$ such that 
\[
{\mathscr S} \left( \bar{\bar \alpha}\right) =
\Theta_{Prep} \left( 
\lim_{n \rightarrow \infty} \Theta_{{\mathcal P}er}
\left( P \right)
\right).
\]
\label{th:fer}
\end{theorem}

\begin{remark}
Theorem \ref{th:6.3} gives an effective generation of 
any stepped surface under the Pisot condition.
\end{remark}

\begin{conjecture}
\label{conclusion}
Let $\bar{\bar \alpha}= \left( \alpha_{0}, \alpha_{1}, \alpha_{2} \right)$ be any $\mathbb{Q}$-basis of arbitrarily given real cubic number field.  Then the stepped surface ${\mathscr S} \left( \bar{\bar \alpha} \right)$ is finitely descriptive, i.e., 
there exist a finite word 
\(
\varepsilon_0 \varepsilon_1 \cdots
\varepsilon_{k-1} \in Ind^{*}= \bigcup_{n=0}^{\infty}
Ind^{n},
\)
and 
a nonempty word
\(
\varepsilon_{k} 
\varepsilon_{k+1} \cdots \varepsilon_{k+l-1} \in Ind^{*}
\)
$ \left( k \geq 0, \ l>1 \right)$,
and a patch $P$ consisting of finite squares such that
\[
{\mathscr S} \left( \bar{\bar \alpha} \right) =
{}_g \Psi_s \left(
\Theta_{\varepsilon _0}
\Theta_{\varepsilon_1}
\cdots
\Theta_{\varepsilon_{k-1}}
\left(
\lim_{n \rightarrow \infty} \left(
\Theta_{\varepsilon_{k}}
\Theta_{\varepsilon_{k+1}}
\cdots
\Theta_{\varepsilon_{k+l-1}}
\right)^{n} \left( P \right)
\right)
\right).
\]

\end{conjecture}

\begin{remark}
Conjecture \ref{conclusion}
says that any stepped surface
${\mathscr S} \left(  \bar{\bar \alpha} \right)$
for any $\mathbb{Q}$-basis 
$ \bar{\bar \alpha}  \in K^{3}$ for any given real cubic number field $K$ is finitely descriptive
and generated only by $6$ substitutions.
For notation of $\Theta_{\varepsilon} \left( \varepsilon \in Ind \right)$, see Section \ref{sec:stepped}.  Notice that 
without loss of generality, we may assume 
$\alpha_{0}, \alpha_{1}, \alpha_{2}>0$ and 
$\alpha_{0}+\alpha_{1}+\alpha_{2} = 1$ by the symmetry of the lattice $\mathbb{Z}^{3}$.
\end{remark}
%

%
\noindent
{\bf Acknowledgements}
The first author is supported by Grant-in-Aid 
for 
Scientific Research (C) No. 22540117.
The second author is supported by Grant-in-Aid 
for 
Scientific Research (C) No. 22540119.
The third author is supported by JST PRESTO program 
and by Grant-in-Aid 
for 
Young Scientists (B) No. 21700256.
The fourth and fifth authors are supported by Grant-in-Aid 
for 
Scientific Research (C) No. 22540037.

\end{document}

%% file: table1.tex
\begin{tabular}{|c|c|}
\hline
$\varepsilon(\alpha_{n},\beta_{n})$&
$\varepsilon(\alpha_{n+1},\beta_{n+1})$\\
\hline
$(1,2)$&$(0,1), (1,0)$\\
\hline
$(2,1)$&$(0,2), (2,0)$\\
\hline
$(0,1)$&$(0,2), (2,0)$\\
\hline
$(1,0)$&$(1,2), (2,1)$\\
\hline
$(0,2)$&$(0,1), (1,0)$\\
\hline
$(2,0)$&$(1,2), (2,1)$\\
\hline
\end{tabular}

%% file: table2.tex
\begin{tabular}{c|c|c|c}
\hline
\text{$(\zeta,\eta)$} 
&
\text{$\varepsilon_{K_1}(\zeta,\eta)$}
&
\text{$T_{K_1}(\zeta,\eta)$}
&
\text{$\varepsilon_{K_1} 
(T_{K_1} (\zeta,\eta) )$}\\
\hline
$(5/39+7\lambda/39-2\lambda^2/39,$&(2,1)&
$(6/37+6\lambda/37-\lambda^2/37,$&(2,1)\\
$2/39-5\lambda/39+7\lambda^2/39)$&
&
$-5/37-5\lambda/37+7\lambda^2/37)$&
\\
\hline
$(6/37+6\lambda/37-\lambda^2/37,$&(2,1) &
$(1/5+\lambda/5,$&(1,0)\\
$-5/37-5\lambda/37+7\lambda^2/37)$&&
$-2/5-\lambda/5+\lambda^2/5)$&
\\
\hline
$(-16/15+\lambda/15+4\lambda^2/15,$&(2,1)&
$(-9/17+\lambda/17+3\lambda^2/17,$&(0,1)\\
$4/5+\lambda/5-\lambda^2/5)$&&
$15/17+4\lambda/17-5\lambda^2/17)$&
\\
\hline
$(4/5+\lambda/5-\lambda^2/5,$&(2,1)&
$(5-\lambda^2,$&(1,2)\\
$-4/15+4\lambda/15+\lambda^2/15)$&&
$-19/3+\lambda/3+4\lambda^2/3)$&
\\
\hline
\end{tabular}

%% file: table3.tex
\begin{tabular}{c|c|c|c}
\hline
\multicolumn{1}{c|}{\text{$(\zeta,\eta)$}}&
\multicolumn{1}{c|}{\text{$\varepsilon_{K_2}(\zeta,\eta)$}}&\multicolumn{1}{c|}{\text{$T_{K_2}(\zeta,\eta)$}}&
\multicolumn{1}{c}{\text{$\varepsilon_{K_2}(T_{K_2}(\zeta,\eta))$}}
\\
\hline
$(1/4+\mu/4-3\mu^2/20,$&(2,1)&
$(-2/17+6\mu/17-\mu^2/17,$&(2,1)\\
$\mu^2/5)$&
&
$8/17-7\mu/17+4\mu^2/17)$&
\\
\hline
$(-2/17+6\mu/17-\mu^2/17,$&(2,1)&
$(-3/26+7\mu/26+\mu^2/26,$&(1,0)\\
$8/17-7\mu/17+4\mu^2/17)$&&
$10/13-6\mu/13+\mu^2/13)$&
\\
\hline
$(-1/6+\mu/6+\mu^2/30,$&(2,1)&
$(-1/11+3\mu/22+\mu^2/22,$&(0,1)\\
$11/12-5\mu/12+7\mu^2/60)$&&
$10/11-4\mu/11+\mu^2/22)$&
\\
\hline
$(1-\mu^2/5,$&
(2,1)&
$(-1+\mu,$&
(1,2)\\
$1/2+\mu/2-3\mu^2/10)$&
&
$-1/2-\mu/2+\mu^2/2)$&
\\
\hline
\end{tabular}